\documentclass[12pt]{article}
\usepackage{amssymb}
\usepackage{amsmath,bm,natbib,graphics,mathrsfs,graphicx,epsfig,subfigure,placeins,indentfirst,setspace,enumerate,enumitem,caption,varioref,booktabs,tabularx,color,epstopdf}
\setcitestyle{authoryear,open={(},close={)}}
\usepackage{amsthm}
\makeatletter \oddsidemargin  -.1in \evensidemargin -.1in
\begin{document}
	\newcommand{\bea}{\begin{eqnarray}}
		\newcommand{\eea}{\end{eqnarray}}
	\newcommand{\nn}{\nonumber}
	\newcommand{\bee}{\begin{eqnarray*}}
		\newcommand{\eee}{\end{eqnarray*}}
	\newcommand{\lb}{\label}
	\newcommand{\nii}{\noindent}
	\newcommand{\ii}{\indent}
	\newtheorem{theorem}{Theorem}[section]
	\newtheorem{example}{Example}[section]
	\newtheorem{corollary}{Corollary}[section]
	\newtheorem{definition}{Definition}[section]
	\newtheorem{lem}{Lemma}[section]
	\newtheorem{remark}{Remark}[section]
	\newtheorem{proposition}{Proposition}[section]
	\numberwithin{equation}{section}
	\renewcommand{\theequation}{\thesection.\arabic{equation}}
	\renewcommand\bibfont{\fontsize{10}{12}\selectfont}
	\setlength{\bibsep}{0.0pt}
		\title{\bf Different informational characteristics of cubic transmuted distributions**}
		\author{Shital {\bf Saha}$^1$\thanks {Email address: shitalmath@gmail.com}, ~Suchandan {\bf Kayal}$^2$\thanks{Corresponding author : Suchandan Kayal (kayals@nitrkl.ac.in,~suchandan.kayal@gmail.com)}, ~and N. {\bf Balakrishnan}$^3$\thanks {Email address: bala@mcmaster.ca
		\newline**It has been accepted on \textbf{Brazilian Journal of Probability and Statistics}.}}
		\date{}
		\maketitle \noindent {\it $^{1,2}$Department of Mathematics, National Institute of
			Technology Rourkela, Rourkela-769008, Odisha, India} \\
		{\it $^{3}$Department of Mathematics and Statistics,
			McMaster University, Hamilton, Ontario L8S 4K1,
			Canada}
			
\date{}
\maketitle
		\begin{center}
	Abstract
		\end{center}
		  Cubic transmuted (CT) distributions were introduced recently by \cite{granzotto2017cubic}. In this article, we derive Shannon entropy, Gini's mean difference and Fisher information (matrix) for CT distributions and establish some of their  theoretical properties. In addition, we  propose cubic transmuted Shannon entropy and cubic transmuted Gini's mean difference. The CT Shannon entropy is expressed in terms of Kullback-Leibler divergences, while the CT Gini's mean difference is shown to be connected with energy distances. We show that the Kullback-Leibler and Chi-square divergences are free of the underlying parent distribution.  Finally, we carry out some simulation studies for the proposed information measures from an inferential viewpoint. 

\section{Introduction}

	Recently, many different generalized families of distributions have been proposed in the literature in order to provide greater flexibility while modelling data arising from diverse applications. In this regard, \cite{shaw2009alchemy} introduced an important class of distributions, known as the class of transmuted distributions, by using a quadratic ranking transmutation map.  Let $F(\cdot)$ be the cumulative distribution function (CDF) of a baseline distribution. Then, the CDF of the transmuted distribution is (\cite{shaw2009alchemy})
			 \begin{eqnarray}\label{eq1.1}
			  F_{X_{T}}(x)=(1+\lambda)F(x)-\lambda F^{2}(x),~~|\lambda|\le 1 ~~x\in\mathbb{R},
			\end{eqnarray}
where $\mathbb{R}$ denotes the set of real numbers.  The transmuted generalized extreme value distribution was studied by \cite{aryal2009transmuted}. They also described an application of the transmuted Gumbel distribution for the climate data. There are some real life data sets whose complexity can not be adequately captured using quadratic transmuted distributions.   A significant amount of work has been made towards developing a new transmuted model and subsequently discussing its  enhanced flexibility in modelling various types of real life data, where the parent model does not provide a good fit. \cite{granzotto2017cubic} established a mixture representation for the transmuted distribution in (\ref{eq1.1}), using distributions of order statistics.	To be specific, with $X_1$ and $X_2$ being i.i.d. random variables with distribution function $F(x) $ and $(X_{1:2},X_{2:2})$ denoting the corresponding order statistics, they considered the mixture random variable
			 \begin{eqnarray}\label{(eq1.2)}
				X\overset{d}{=}
				 \left\{
				 \begin{array}{ll}
				  X_{1:2}  ~\mbox{with~probability}~ ~\pi\\
				 X_{2:2}~ \mbox{with~probability}~ 1-\pi,
				 \end{array}
				 \right.
				 \end{eqnarray}	
 where $0\leq\pi\leq1$	is the mixing probability. Then, using the CDFs of $X_{1:2}$ and $X_{2:2}$ given by $1-\{1-F(x)\}^2$ and $F^2(x)$	(see, for example, \cite{arnold1992first}), respectively,	(\ref{(eq1.2)}) readily yields the CDF of $X$ to be
 	\begin{eqnarray}\label{(eq1.3)}
	\nonumber
	F_X(x)&=& \pi \big[1-\{1-F(x)\}^2\big]+(1-\pi)F^2(x)\\
	&=&2\pi F(x)+(1-2\pi)F^2(x).			
   \end{eqnarray}	
   It is evident that the quadratic transmuted distribution in (\ref{eq1.1}) is a simple reparametrization of the mixture distribution in (\ref{(eq1.3)}) (with $2\pi=1+\lambda$).

   Then, proceeding along the same lines and considering $X_1, X_2$ and $X_3$ to be i.i.d. random variables with distribution function $F(x)$ and $(X_{1:3},X_{2:3}, X_{3:3})$ as the corresponding order statistics, \cite{granzotto2017cubic} considered the mixture random variable	
		\begin{eqnarray}\label{(eq1.4)}
					X\overset{d}{=}
					 \left\{
					 \begin{array}{ll}
					  X_{1:3}  ~\mbox{with~probability}~ ~\pi_1\\
					 X_{2:3}~ \mbox{with~probability}~ \pi_2\\
					  X_{3:3}~ \mbox{with~probability}~ \pi_3,
					 \end{array}
					 \right.
					 \end{eqnarray}	
where $\pi_1, \pi_2, \pi_3 \in (0,1)$ are the mixing probabilities such that  $\pi_1+ \pi_2+ \pi_3=1$. The CDFs and PDFs of $X_{1:3},$ $X_{2:3}$, $X_{3:3}$ are denoted by $F_{\min}(\cdot)$, $F_{2:3}(\cdot),$ $F_{\max}(\cdot)$ and $f_{\min}(\cdot)$, $f_{2:3}(\cdot),$ $f_{\max}(\cdot)$, respectively. Then, using the CDFs of $X_{1:3},X_{2:3}$ and $X_{3:3}$ given in \cite{arnold1992first}, for example, (\ref{(eq1.4)}) readily yields the CDF of $X$ to be  					
	\begin{eqnarray}\label{(eq1.5)}
		F_X(x)&=& 3\pi_1F(x)+3(\pi_2-\pi_1)F^2(x)+(1-3\pi_2)F^3(x). 			
	   \end{eqnarray}		
Upon re-parametrizing the distribution in (\ref{(eq1.5)}) by setting $3\pi_1=\lambda_1$ and $3\pi_2=\lambda_2$, \cite{granzotto2017cubic} proposed the cubic transmuted distribution with CDF
\begin{eqnarray}\label{(eq1.6)}
		F_{X_{CT}}(x)= \lambda_1 F(x)+(\lambda_2-\lambda_1)F^2(x)+(1-\lambda_2)F^3(x), 			
	   \end{eqnarray}
where $\lambda_1\in[0,1]$ and $\lambda_2\in[-1,1]$. The density function corresponding to the CDF in (\ref{(eq1.6)}) is
\begin{eqnarray}\label{eq1.7}
		f_{X_{CT}}(x)=  f(x)\{\lambda_1+2(\lambda_2-\lambda_1)F(x)+3(1-\lambda_2)F^{2}(x)\}.			
	   \end{eqnarray}			
 The model considered by \cite{rahman2018general} is a simple re-parametrization of the distribution in (\ref{(eq1.6)}). Observe that if $\lambda_1$ is taken to be $1+\lambda$ and $\lambda_2$ to be 1 in (\ref{(eq1.6)}), then the quadratic transmuted distribution in (\ref{eq1.1}) is deduced. Further, if  $\lambda_1$ is taken to be $1+\lambda$ and $\lambda_2$ to be $1-\lambda$ in (\ref{(eq1.6)}), then the cubic transmuted distribution with single parameter is deduced (see \cite{al2017cubic}). The corresponding CDF and PDF of the single parameter CT distribution are, respectively, given by
	\begin{eqnarray}\label{eq1.8}
			    	F^*_{X_{CT}}(x)=(1+\lambda)F(x)-2\lambda F^2(x)+\lambda F^3(x)
				\end{eqnarray}
				and
				\begin{eqnarray}\label{eq1.9}
				 	f^*_{X_{CT}}(x)=f(x)\left\{ (1+\lambda)-4\lambda f(x)+3\lambda f^2(x)\right\},
			  \end{eqnarray}
			  where $|\lambda|\le 1$ and $x\in \mathbb{R}$.

\cite{aslam2018cubic} studied goodness of fit tests that the cubic transmuted distribution fits better than the parent distributions for the real data set (see also \cite{rahman2020cubic} and \cite{ahsan2023new}).
Later, \cite{rahman2019cubic} showed that cubic transmuted power distribution fits better than the parent distribution for some real data set.   It is important to mention here that such mixture distributions of order statistics had earlier been considered in the pioneering work of \cite{lehmann1953} as a nonparametric alternative in his development of rank based tests (see also the related comments by \cite{balakrishnan2021my}). 

The uncertainty associated with a probability model can be computed using several information measures. In many situations related to lifetime data analysis, experimental physics, econometrics and demography, uncertainty quantification in a distribution is very important. The uncertainty in a random variable can be measured using the concepts of Fisher information (see \cite{fisher1929tests}), Shannon entropy (see \cite{shannon1948mathematical}) and R\'enyi entropy (see \cite{renyi1961measures}). Besides these, various divergence measures such as Kullback-Leibler (KL) divergence (see \cite{kullback1951information}), Chi-square divergence  (see \cite{nielsen2013chi}) have been introduced to measure the closeness between two probability distributions.
		
		From the above discussion,  we note that although many transmuted families of distributions have been introduced and studied by several researchers, these models have not been explored from information theoretical viewpoint except for the work of \cite{kharazmi2021informational} . \cite{kharazmi2021informational} studied transmuted distributions and established several theoretic informational properties.  They also provided extensions of  Shannon entropy and Gini's mean difference.  Motivated by the work of \cite{kharazmi2021informational}, we extend here their results to the results for the family of cubic transmuted distributions partially.

		In this paper, we explore several informational properties of a more flexible and general class of cubic transmuted distributions in (\ref{(eq1.6)}). We study Shannon entropy, Gini's mean difference, and Fisher information of the class of cubic transmuted distributions. In addition, we also study KL divergence, Chi-square divergence, and energy distance divergence (\cite{cramer1928composition}) between CT distribution and its components. We propose  extensions of Shannon entropy and Gini's mean difference and call these as CT Shannon entropy and CT Gini's mean difference, respectively.
		
		The rest of this paper proceeds as follows. In Section $2$, we study Shannon entropy of a CT distribution. We obtain Kullback-Leibler divergence between CT distribution and the densities of its components. It is noted that the KL divergence between a general transmuted model and its components is free of the parent distribution. The cubic transmuted Shannon entropy is proposed as an extension of Shannon entropy of CT distribution and Jensen-Shannon (JS) divergence measure. Further, the proposed CT Shannon entropy is expressed in terms of  KL divergence measures of the general model and its components. Gini's mean difference of the CT distributions is proposed and studied in Section $3.$ Further, Gini's mean difference of the CT distribution is expressed as the sum of Gini's mean difference of the general transmuted distribution and that of its components. Similar to the CT Shannon entropy, CT Gini's mean difference is proposed. In Section $4$, we discuss the Chi-square divergence between the general transmuted distribution and the density functions of its components. Section $5$ presents Fisher information for the CT distribution. We obtain Fisher information matrix for the CT distribution given in (\ref{(eq1.6)}). In addition, we also present Fisher information of the particular model with CDF as in (\ref{eq1.8}).   Section $6$ describes simulation investigation to show the importance of the proposed information measures from an inferential viewpoint. Finally, some concluding remarks are added in Section $7.$
		
 \section{Shannon entropy, KL divergence and CT Shannon entropy}
				The purpose of this section is three-fold. First, we study Shannon entropy of the CT distribution. Then, we derive KL divergences for the general CT distribution and its components, and finally we define CT Shannon entropy, which can be treated as an extension of the Shannon entropy and JS divergence.
				\subsection{Shannon entropy}
				 Here, we derive Shannon entropy of the CT distribution with CDF and PDF given in (\ref{(eq1.6)}) and (\ref{eq1.7}), respectively.
				Let $X$ be an absolutely continuous random variable with PDF $f(\cdot)$. Then, the Shannon entropy and weighted Shannon entropy of $X$ are, respectively, given by
				\begin{eqnarray}\label{eq2.1*}
				  H(f)=-\int_{-\infty}^{\infty}f(x)\log \big(f(x)\big) dx~~\mbox{and}~~ H^{\psi}(f)=-\int_{-\infty}^{\infty}\psi(x)f(x)\log \big(f(x)\big) dx,
				\end{eqnarray}
				    where $\psi$ is a nonnegative real-valued measurable function. We note that the concept of weighted entropy takes into account values of different outcomes and this makes entropy context-dependent, through the choice of weight function. In particular, when $\psi(x)=x,$ the functional $H^{\psi}(f)$ reduces to a length-biased information measure (see \cite{di2007weighted}). A random variable $X$ is said to have beta distribution if its PDF is of the  form
				    \begin{eqnarray}\label{eq2.6}
				    g(x)=\frac{x^{a-1}(1-x)^{b-1}}{\mathbb{B}(a,b)},~~0<x<1,~a,b>0,
				    \end{eqnarray}   		
				    where $\mathbb{B}(a,b)$ denotes the complete beta function. Henceforth, we denote $X\sim Beta(a,b)$ if its PDF is as in (\ref{eq2.6}).		
			  		
	\begin{theorem}\label{th2.1}
	  	Let $X_1$, $X_2,$ and $X_3$ be i.i.d. random variables with a common CDF $F(\cdot)$ and PDF $f(\cdot)$. Then, the Shannon entropy of a CT random variable $X_{CT}$ is given by
	  	\begin{eqnarray*}
	   H(f_{X_{CT}})&=&\lambda_1H(f)+(1-\lambda_2)H(f_{\max})+2(\lambda_2-\lambda_1)H^F(f)-(1-\lambda_2)H(f_W)+H(f_{U_{CT}}),
	  	\end{eqnarray*}
	  	where $U_{CT}$ is the CT uniform random variable, $W\sim Beta(3,1)$, and $H^F(X)$ means weighted Shannon entropy defined in (\ref{eq2.1*}) with weight function $F.$
	  \end{theorem}    	
	
	\begin{proof}
		The Shannon entropy of the CT random variable $X_{CT}$ is
		\begin{eqnarray}\label{eq2.2}
		H(f_{X_{CT}})=-\int_{-\infty}^{\infty}f_{X_{CT}}(x)\log \big(f_{X_{CT}}(x)\big)dx
		=I_{1}+I_{2},
		\end{eqnarray}
		where
		\begin{eqnarray*}
			I_1&=& -\int_{-\infty}^{\infty}    \big[ f(x)\{\lambda_{1}+2(\lambda_{2}-\lambda_{1})F(x)
			+3(1-\lambda_{2})F^2(x)\}\big] \log \big(f(x)\big)dx,\\
			I_2&=&-\int_{-\infty}^{\infty}    \big[ f(x)\{\lambda_{1}+2(\lambda_{2}-\lambda_{1})F(x)+3(1-\lambda_{2})F^2(x)\}\big] \log \big(\lambda_{1}+2(\lambda_{2}-\lambda_{1})F(x)\\&~&
			+~3(1-\lambda_{2})F^2(x)\big)dx.
		\end{eqnarray*}
	Using the transformation $u=F(x)$ in (\ref{eq2.2}) and after some algebraic calculations, we obtain
		\begin{eqnarray*}
	H(f_{X_{CT}})&=&\lambda_1H(f)+(1-\lambda_2)H(f_{\max})+(1-\lambda_2)\int_{0}^{1}3u^2\log(3u^2)du\\
	&~&-\int_{0}^{1} \big\{ \lambda_{1}+2(\lambda_{2}-\lambda_{1})u+3(1-\lambda_{2})u^2\big\} \log \big(\lambda_{1}+2(\lambda_{2}-\lambda_{1})u\\
	&~&+3(1-\lambda_{2})u^2\big)du-2(\lambda_{2}-\lambda_{1})\int_{-\infty}^{\infty}F(x)f(x)\log \big(f(x)\big)dx\\
	&=&\lambda_1H(f)+(1-\lambda_2)H(f_{\max})-(1-\lambda_2)H(f_W)+H(f_{U_{CT}})\\&~&+2(\lambda_2-\lambda_1)H^F(f).
	\end{eqnarray*}	
	  Hence, the result.
	  \end{proof}
	  Substituting $\lambda_1=1+\lambda$ and $\lambda_2=1$ in the result of Theorem  \ref{th2.1}, we obtain
	  \begin{eqnarray}\label{eq2.3}
	  H(f_{X_{T}})=(1+\lambda)H(f)+H(f_{U_{T}})-2\lambda H^{F}(f),
	  \end{eqnarray}
	  where $f_{U_{T}}(\cdot)$ is the PDF of the transmuted uniform random variable. Now,
	  \begin{eqnarray}\label{eq2.4}
	  -2\lambda H^{F}(f)&=&2\lambda\int_{-\infty}^{\infty}F(x) f(x) \log \big(f(x)\big) dx\nonumber\\
	  &=& 2 \lambda_1 \int_{-\infty}^{\infty}F(x) f(x) \log \Big(\frac{2F(x)f(x)}{2F(x)}\Big) dx\nonumber\\
	  &=&\lambda \int_{-\infty}^{\infty}2 F(x) f(x) \log \big(2 F(x) f(x)\big)dx-\lambda \int_{-\infty}^{\infty}2F(x) f(x) \log \big(2 F(x)\big)dx\nonumber\\
	  &=& -\lambda H(f_{\max})-\lambda\int_{0}^{1} 2u \log (2u)du ~~\big(\mbox{using}~ u=F(x)\big)\nonumber\\
	  &=&-\lambda H(f_{\max}) +\lambda H(f_{V}),
	  \end{eqnarray}
	 where $f_{V}(\cdot)$ is the PDF of $V\sim Beta(2,1).$ In the following remark, we study Shannon entropy for two special cases of CT distribution.
	  	  	
	\begin{remark}~~
	  	  			\begin{itemize}
	  	  		  \item For $\lambda_2=1$ and $\lambda_1=1+\lambda$, using (\ref{eq2.4}), the Shannon entropy of CT distribution becomes
	  	  		  \begin{eqnarray*}
	  	  		  H(f_{X_{T}})=(1+\lambda)H(f)-\lambda H(f_{\max})+\lambda H(f_V)+H(f_{U_{T}}),
	  	  		 \end{eqnarray*}
	  	  		  where $U_T$ is the transmuted uniform random variable and $V\sim Beta(2,1)$. For the details, one may refer to \cite{kharazmi2021informational}.
	  	  		
	  	  		\item  For $\lambda_1=1+\lambda$ and $ \lambda_2=1-\lambda$,  the expression of Shannon entropy for the CT distribution in Theorem \ref{th2.1} reduces to the Shannon entropy of one parameter CT distribution with CDF in (\ref{eq1.8}), which is given by
	  	  		   \begin{eqnarray*}
	  	  		    H(f^*_{X_{CT}})=(1+\lambda)H(f)+\lambda H(f_{\max})-2\lambda H^{F}(f)-\lambda H(f_W)+H(f^*_{U_{CT}}).
	  	  		    \end{eqnarray*}
	  	  		
	  	  	      \end{itemize}
	  	  		   \end{remark}			

 \begin{remark}\label{re2.2}
	  	Define $\theta(\lambda_1,\lambda_2)=H(f_{U_{CT}})-(1-\lambda_2)H(f_W)$. Then, after some algebraic calculations, we obtain a closed-form expression for $\theta(\lambda_1,\lambda_2)$ as
	  	\begin{eqnarray*}
	  	\theta(\lambda_1,\lambda_2)&=&-\Big\{
	  	\log(3-\lambda_1-\lambda_2)-\phi(p)+\phi(q)-\frac{2}{3}(3+\lambda_1-4\lambda_2)\\
	  	&~&+\frac{2}{9(1-\lambda_2)}(\lambda^2_1+13\lambda^2_2-2\lambda_1\lambda_2-18\lambda_2+6)\Big\}-(1-\lambda_2)\Big\{\frac{2}{3}-\log(3)\Big\},
	  	\end{eqnarray*}
	  	where
	  	\begin{eqnarray*}
	  	\phi(x)&=&\frac{1}{2r}\left\{2(\lambda_2-\lambda_1)^2 x^2+10 (\lambda_2 -\lambda_1)(1-\lambda_2)x^3+12 (1-\lambda_2)^2 x^4\right\}\log \Big(\frac{x}{x-1}\Big),
	  \\
	  	r&=&\sqrt{\lambda_1^2+\lambda_2^2+\lambda_1 \lambda_2-3\lambda_1},~
	  	p=\frac{\lambda_1-\lambda_2+r}{3(1-\lambda_2)},~
	  	q=\frac{\lambda_1-\lambda_2-r}{3(1-\lambda_2)}.
	  	\end{eqnarray*}
	  Because of the complex about its form of $\theta(\lambda_1,\lambda_2)$, it is difficult to analyse it from a theoretical viewpoint. To get an idea about its codomain, we plot the function with respect to $\lambda_1$ when $\lambda_2$ is fixed, and with respect to $\lambda_2$ when $\lambda_1$ is fixed (see Figure \ref{fig1}). From the figure, it is clearly seen that $\theta(\lambda_1,\lambda_2)$ takes both positive and negative values.
	 	\end{remark}
	

	\begin{figure}[htbp!]
		  	\centering
		  	\includegraphics[width=13.5cm,height=8cm]{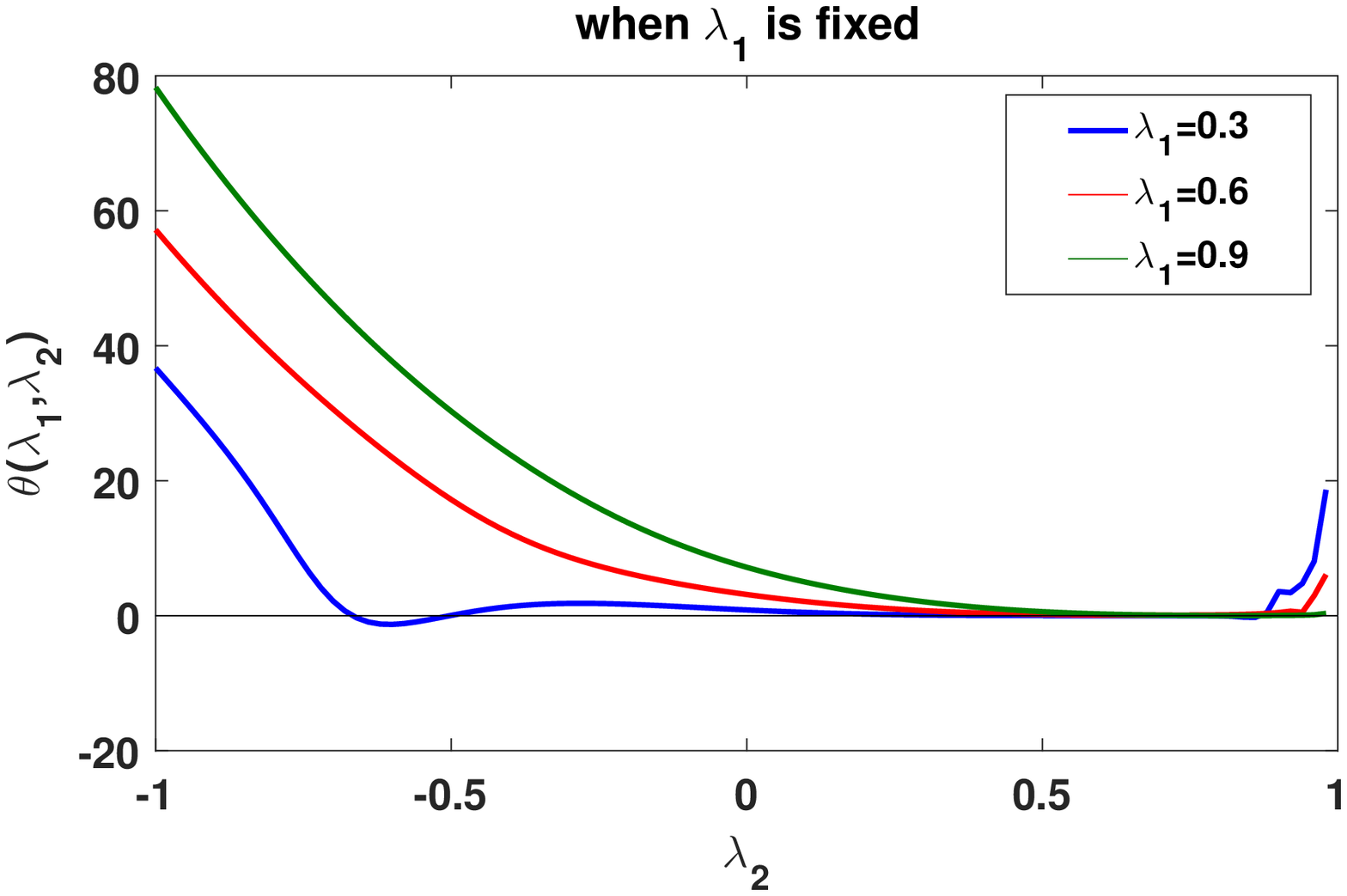}
		  	\includegraphics[width=13.5cm,height=8cm]{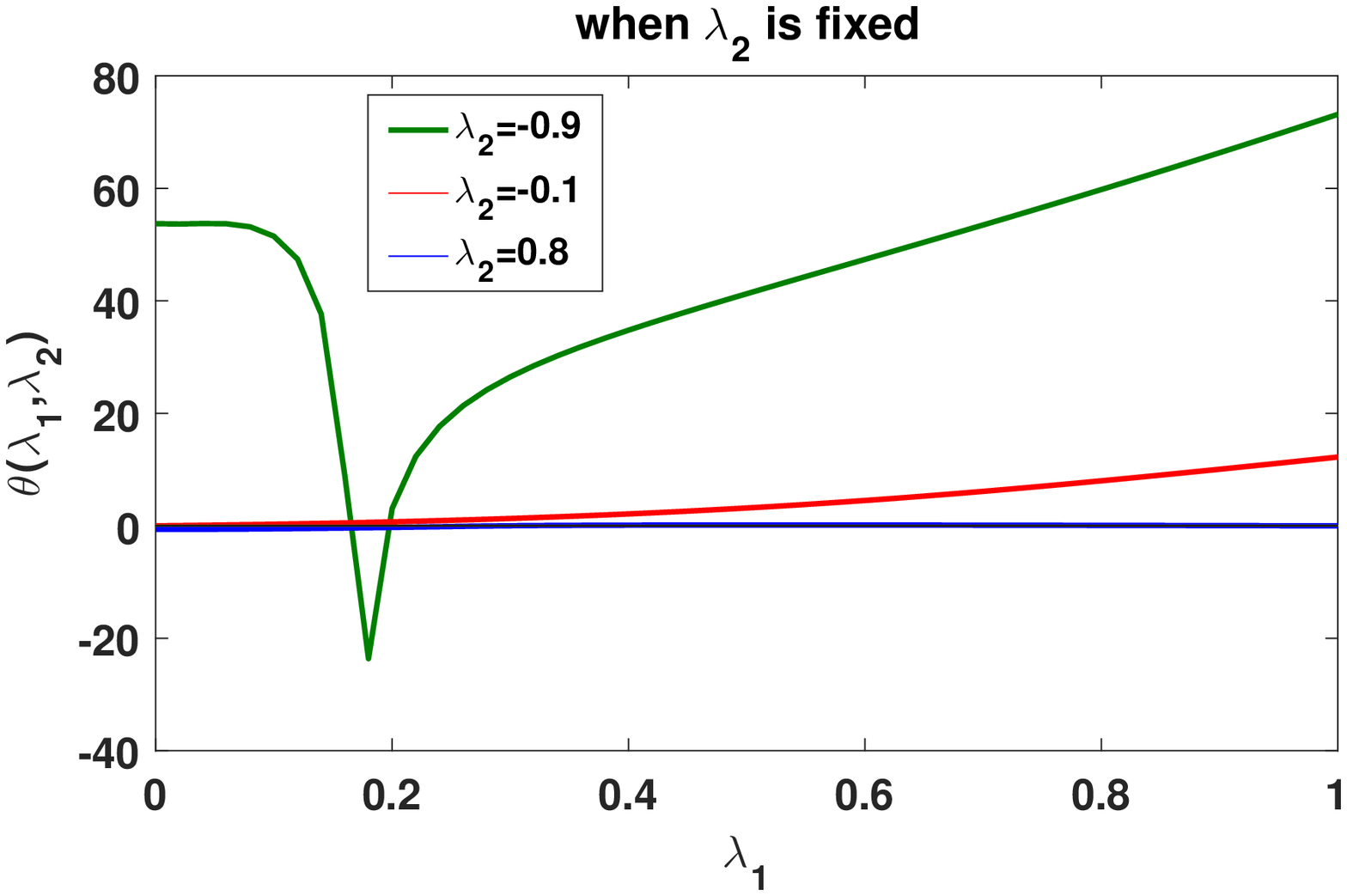}
		  	\caption{Plots of $\theta(\lambda_1,\lambda_2)$ for different choices of the parameters.}
		  	\label{fig1}
		  \end{figure}

 Next, we consider the following example to illustrate the result in Theorem \ref{th2.1}.	
		  \begin{example}
		  	Let $X_1$, $X_2,$ and $X_3$ be i.i.d.  Pareto variables with PDF $f(x)=1-x^{-\alpha}, ~x\geq1, ~\alpha>0$. Then, the Shannon entropy for the CT random variable $X_{CT}$ is given by
		  	\begin{eqnarray*}
		  		H(f_{X_{CT}})
		  		 &=&\lambda_1\Big\{\frac{\alpha+1}{\alpha^2}-\log(\alpha)\Big\}+(1-\lambda_2)\Big\{\frac{1}{6}-\frac{7}{6\alpha}-\log(3\alpha)\Big\}\\&~&+(\lambda_2-\lambda_1)\Big\{\frac{1}{2\alpha}(\alpha+1)(3-2\alpha)-\log(\alpha)\Big\}
		  		+\theta(\lambda_1,\lambda_2),
		  	\end{eqnarray*}
		   where $\theta(\lambda_1,\lambda_2)$ is as given in Remark \ref{re2.2}.
		  \end{example}
		
	  \begin{example}
		  	Let $X_1$, $X_2$ and $X_3$ be i.i.d.   exponential variables with PDF $f(x)=\beta e^{-\beta x}$, $x>0,$~$\beta>0$. Then, the Shannon entropy of  $X_{CT}$ is given by
		  	\begin{eqnarray*}
		  		H(f_{X_{CT}})
		  		=\lambda_1\big\{1-\log(\beta)\big\}+(1-\lambda_2)\Big\{\log(3\beta)-\frac{5}{2}\Big\}-(\lambda_2-\lambda_1)\log(\beta)+\theta(\lambda_1,\lambda_2),
		  	\end{eqnarray*}
		  	where $\theta(\lambda_1,\lambda_2)$ is as given in Remark \ref{re2.2}.
		  \end{example}
		  		
\subsection{Kullback-Leibler divergence}
  We prove the following result which shows that Kullback-Leibler divergence measures between CT distribution and its components are equal to those between CT uniform distribution and its components. Let  $f_1(\cdot)$ and $f_2(\cdot)$ be the PDFs of two absolutely continuous random variables, say $X_{1}$ and $X_2$, respectively. Then, the Kullback-Leibler divergence between $f_1(\cdot)$ and $f_2(\cdot)$ is given by (see \cite{kullback1951information})
  \begin{eqnarray}
  KL(f_1,f_2)=\int_{-\infty}^{\infty}f_{1}(x)\log\left(\frac{f_{1}(x)}{f_{2}(x)}\right)dx.
  \end{eqnarray}
  Note that the Kullback-Leibler divergence is not symmetric, that is, in general, $KL(f_1,f_2)\ne KL(f_2,f_1)$. Further, it can be established that Kullback-Leibler divergence takes non-negative values, and that it could it may be infinity.
  	\begin{theorem}\label{th2.1*}
  	Suppose $X$ is a random variable with PDF $f(\cdot)$. Further, let the CT random variable $X_{CT}$ have PDF $f_{X_{CT}}(\cdot)$.  Then,
  	\begin{itemize}
  	\item [(a)]  $KL(f,f_{X_{CT}}) = KL(f_U,f_{U_{CT}}),$
  	\item [(b)]  $KL(f_{X_{CT}},f) = KL(f_{U_{CT}},f_U),$	
  	\end{itemize}
  	where uniform and CT uniform random variables, denoted by $U$ and $U_{CT}$ in the interval $(0,1)$, have PDFs $f_U(\cdot)$ and $f_{U_{CT}}(\cdot)$, respectively.
  \end{theorem}

\begin{proof}
	$(a)$ The Kullback-Leibler divergence between $f(\cdot)$ and $f_{X_{CT}}(\cdot)$ is given by
  \begin{eqnarray*}
  	KL(f,f_{X_{CT}})&=&
  	\int_{-\infty}^{\infty}f(x)\log\bigg(\dfrac{f(x)}{f_{X_{CT}}}\bigg)dx\\
  	&=& \int_{-\infty}^{\infty}f(x)\log\bigg(\dfrac{f(x)}{f(x)\big\{\lambda_1+2(\lambda_2-\lambda_1)F(x)+3(1-\lambda_2)F^2(x)\big\}}\bigg)dx\\
  	&=&-\int_{0}^{1}\log\big(\lambda_1+2(\lambda_2-\lambda_1)u+3(1-\lambda_2)u^2\big)du~\big(\mbox{using} ~u=F(x)\big)\\
  	&=&KL(f_U,f_{U_{CT}}).
  \end{eqnarray*}
  $(b)$ The second part can be proved similarly to the first part, and is therefore omitted.
  \end{proof}

\begin{theorem}\label{th2.2*}
	Suppose $X_1$, $X_{2}$ and $X_3$ are i.i.d. as $X$ with PDF $f(\cdot)$.  Further, let $f_{2:3}(\cdot)$ be the PDF of $X_{2:3}$. Then,
	\begin{itemize}
	\item[(a)] $KL(f_{2:3},f_{X_{CT}}) = KL(f_{U_{2:3}},f_{U_{CT}}),$
	\item[(b)] $KL(f_{2:3},f_{\max}) = KL(f_{U_{2:3}},f_W),$
	\end{itemize}
	where $f_W(\cdot)$ is the PDF of $W\sim Beta(3,1)$,
	 $U_{CT}$ is as defined in Theorem \ref{th2.1*}, and $f_{U_{2:3}}(\cdot)$ is the PDF of $U_{2:3}$.
\end{theorem}

\begin{proof}
$(a)$ The Kullback-Leibler divergence between $f_{2:3}(\cdot)$ and $f_{X_{CT}}(\cdot)$ is given by
\begin{eqnarray*}
	KL(f_{2:3},f_{X_{CT}})&=& \int_{-\infty}^{\infty}f_{2:3}(x)\log\bigg(\dfrac{f_{2:3}(x)}{f_{X_{CT}}(x)}\bigg)dx\\
	&=& \int_{-\infty}^{\infty}f_{2:3}(x)\log\bigg(\dfrac{6f(x)\{F(x)-F^2(x)\}}{f(x)\big\{\lambda_1+2(\lambda_2-\lambda_1)F(x)+3(1-\lambda_2)F^2(x)\big\}}\bigg)	dx\\
	&=& \int_{0}^{1}6(u-u^2)\log\bigg(\dfrac{6(u-u^2)}{\lambda_1+2(\lambda_2-\lambda_1)u+(1-\lambda_2)u^2}\bigg)du~\big(\mbox{using} ~u=F(x)\big)\\
	&=& KL(f_{U_{2:3}},f_{U_{CT}}).
\end{eqnarray*}
$(b)$ The Kullback-Leibler divergence between $f_{2:3}(\cdot)$ and $f_{\max}(\cdot)$ is
\begin{eqnarray*}
	KL(f_{2:3},f_{\max})&=& \int_{-\infty}^{\infty}f_{2:3}(x)\log\bigg(\dfrac{f_{2:3}(x)}{f_{\max}(x)}\bigg)dx\\
	&=&\int_{-\infty}^{\infty}6f(x)\{F(x)-F^2(x)\}\log\bigg(\dfrac{6f(x)\{F(x)-F^{2}(x)\}}{3f(x)F^2(x)}\bigg)dx\\
	&=&\int_{0}^{1}6(u-u^2)\log\bigg(\dfrac{6(u-u^2)}{3u^2}\bigg)du~~\big(\mbox{using} ~u=F(x)\big)\\
	&=& KL(f_{U_{2:3}},f_W).
\end{eqnarray*}
Thus, the proof of the theorem gets completed.
\end{proof}

\begin{theorem}\label{th2.3}
	Suppose $X_1$, $X_{2}$, and $X_3$ are i.i.d. as $X$ with PDF $f(\cdot)$.  Then,
	\begin{itemize}
	\item[(a)] $KL(f_{X_{CT}},f_{\max})= KL(f_{U_{CT}},f_W),$
	\item[(b)] $KL(f_{\max},f_{X_{CT}})= KL(f_W,f_{U_{CT}}),$
	\item[(c)] $KL(f_{\max},f)= KL(f_W,f_U),$
	\end{itemize}
	where $f_{U_{CT}}(\cdot)$ and $f_W(\cdot)$ as given in Theorems \ref{th2.1*} and Theorem \ref{th2.2*}, respectively.
\end{theorem}

\begin{proof}
$(a)$ We have
\begin{eqnarray*}
	KL(f_{X_{CT}},f_{\max})&=& \int_{-\infty}^{\infty}f_{X_{CT}}(x)\log\bigg(\dfrac{f_{X_{CT}}(x)}{f_{\max}(x)}\bigg)dx\\
	&=& \int_{-\infty}^{\infty}f(x)\big\{\lambda_1+2(\lambda_2-\lambda_1)F(x)+3(1-\lambda_2)F^2(x)\big\}\\
	&~&\times \log\bigg(\dfrac{f_{X_{CT}}(x)}{3f(x)F^2(x)}\bigg)dx\\
	 &=&\int_{0}^{1}\big\{\lambda_1+2(\lambda_2-\lambda_1)u+3(1-\lambda_2)u^2\big\}\\&~&\times\log\bigg(\dfrac{\big\{\lambda_1+2(\lambda_2-\lambda_1)u+3(1-\lambda_2)u^2\big\}}{3u^2}\bigg)du~~\big(\mbox{using} ~u=F(x)\big)\\
	&=& KL(f_{U_{CT}},f_W).
\end{eqnarray*}
The proofs for $(b)$ and $(c)$ are similar, and are therefore omitted.
\end{proof}

Next, we present closed-form expressions for the KL divergences between $f_{U_{CT}}(\cdot)$ with $f_{U}(\cdot)$, $f_W(\cdot)$, and $f_{U_{2:3}}(\cdot)$. The detailed derivations are omitted for the sake of brevity.

\begin{eqnarray*}
 (a)~	KL(f_U,f_{U_{CT}})&=& 2-\varphi(p)+\varphi(q)-\log(3-\lambda_1-\lambda_2),\\
 (b)~	KL(f_{U_{CT}},f_U)&=&
 	 \log(3-\lambda_1-\lambda_2)-\phi(p)+\phi(q)-\frac{2}{3}(3+\lambda_1-4\lambda_2)\\&~&+\frac{2}{9(1-\lambda_2)}\big(\lambda^2_1+13\lambda^2_2-2\lambda_1\lambda_2-18\lambda_2+6\big),\\
 (c)~	KL(f_{U_{2:3}},f_{U_{CT}})&=& \log\Big(\dfrac{6}{3-\lambda_1-\lambda_2}\Big)-\frac{5}{3}-\eta(p)+\eta(q)\\&~&-\frac{1}{9(1-\lambda_2)^2}\big(55\lambda^2_2-8\lambda^2_1-44\lambda_1\lambda_2-78\lambda_2+60\lambda_1+15\big),\\
 (d)~	KL(f_{U_{CT}},f_W)&=& -\phi(p)+\phi(q)
 	-\log\Big(\frac{3}{3-\lambda_1-\lambda_2}\Big)+\frac{1}{3}(\lambda_1+9\lambda_2-4)\\
 	&~&+\frac{2}{9(1-\lambda_2)}(\lambda^2_1+13\lambda^2_2-2\lambda_1\lambda_2-18\lambda_2+6),\\
 (e)~	KL(f_W,f_{U_{CT}})&=& \gamma(q)-\gamma(p)+\log\Big(\frac{3}{3-\lambda_1-\lambda_2}\Big)+\frac{(\lambda_2-\lambda_1)}{9(1-\lambda_2)^2}(3-7\lambda_2+4\lambda_1)\\&~&+\frac{2}{9(1-\lambda_2)^2}(4\lambda^2_1+10\lambda^2_2-8\lambda_1\lambda_2-9\lambda_2+3)-\frac{2}{3},\\
 (f)~	KL(f_{U_{2:3}},f_W)&=&\log(2),\\
 (g)~	KL(f_W,f_U)&=&\log(3)-\frac{2}{3},
 \end{eqnarray*}
where $p$, $q$ and $\phi(\cdot)$ are as given in Remark \ref{re2.2}, and
\begin{eqnarray*}
\varphi(x)&=&\frac{1}{r}\{(\lambda_2-\lambda_1)x+3(1-\lambda_2)x^2\}\log\Big(\frac{x}{x-1}\Big),\\
				\eta(x)&=&\frac{1}{2r}\{6(\lambda_2-\lambda_1)x^2-(18+4\lambda_1-22\lambda_2)x^3+12(1-\lambda_2)x^4\}\log\Big(\frac{x}{x-1}\Big),\\
				\gamma(x)&=&\frac{1}{2r}\{2(\lambda_2-\lambda_1)x^3+6(1-\lambda_2)x^4\}\log\Big(\frac{x}{x-1}\Big).
\end{eqnarray*}

 From the above expressions, it is clear that the KL divergences $KL(f_{X_{CT}},f)$, $KL(f,f_{X_{CT}})$, $KL(f_{2:3},f_{X_{CT}})$, $KL(f_{2:3},f_{\max})$, $KL(f_{X_{CT}},f_{\max})$, $KL(f_{\max},f_{X_{CT}})$, $KL(f_{\max},f)$ are all free from the underlying parent distribution with CDF $F(\cdot)$.

 Recall that the KL divergence is non-symmetric. However, one may get symmetric divergence based on KL divergence, known as Jeffreys' divergence. The Jeffreys' divergence between two density functions $f_1(\cdot)$ and $f_2(\cdot)$ is defined as (see \cite{jeffreys1946invariant})
\begin{eqnarray}
J_d(f_1,f_2)=\int_{-\infty}^{\infty}\{f_1(x)-f_2(x)\}\log\Big(\frac{f_1(x)}{f_2(x)}\Big)dx=KL(f_1,f_2)+KL(f_2,f_1).
\end{eqnarray}
Note that the Kullback-Leibler divergence is not symmetric, but  Jeffreys' divergence is symmetric, that is,
\begin{eqnarray*}
J_d(f_1,f_2)=J_d(f_2,f_1).
\end{eqnarray*}
In the following remark, it is shown that Jeffreys' divergences between $f_{X_{CT}}(\cdot)$, $f_{\max}(\cdot)$, and $f_{X_{CT}}(\cdot)$, $f(.)$ are equal to that between $f_{U_{CT}}(\cdot)$, $f_W(\cdot)$, and $f_{U_{CT}}(\cdot)$, $f_U(\cdot)$, respectively. This statement can be proved using the results presented in Theorem \ref{th2.1*} and Theorem \ref{th2.3}.

\begin{remark}
We have
\begin{eqnarray*}
	J_d(f_{X_{CT}},f_{\max})&=&J_d(f_{U_{CT}},f_W)=KL(f_W,f_{U_{CT}})+KL(f_{U_{CT}},f_W),\\
	J_d(f_{X_{CT}},f)&=&J_d(f_{U_{CT}},f_U)=KL(f_U,f_{U_{CT}})+KL(f_{U_{CT}},f_U).
\end{eqnarray*}
Note that like Kullback-Leibler divergence, here, Jeffreys' divergence is also free from the underlying distribution $F(\cdot).$
\end{remark}

Now, we obtain Kullback-Leibler divergence between a mixture distribution and CT distribution in the following example.
\begin{example}
	Suppose the PDF of a general mixture distribution is $f_{\text{mix}}(x)=vf(x)+3(1-v)f(x)F^2(x)$, $v\in[0,1]$. The Kullback-Leibler divergence between $f_{mix}(\cdot)$ and $f_{X_{CT}}(\cdot)$ and $f_{X_{CT}}(\cdot)$ and $f_{mix}(\cdot)$ are, respectively, obtained as
	\begin{eqnarray*}
	(a)~KL(f_{mix},f_{X_{CT}})&=&\log \Big(\frac{3-2v}{3-\lambda_1-\lambda_2}\Big)+\frac{(1-v)}{9(1-\lambda_2)^2}(4\lambda^2_1+13\lambda^2_2-5\lambda_1\lambda_2-3\lambda_1\\&~&-15\lambda_2+6) +\frac{4v^\frac{3}{2}}{3\sqrt{3}\sqrt{1-v}}\tan^{-1}\bigg(\sqrt{\frac{3(1-v)}{v}}\bigg)-\frac{2}{3}(1-2v)\\
	&~&-\xi(p)+\xi(q),\\
	 (b)~KL(f_{X_{CT}},f_{mix})&=&\log\Big(\frac{3-\lambda_1-\lambda_2}{3-2v}\Big)-\frac{1}{3}(4+5\lambda_1-9\lambda_2)-\phi(p)+\phi(q)\\
	 &~&-\frac{v}{3(1-v)}\bigg\{2(1-\lambda_2)-(\lambda_2-\lambda_1)\log\Big(\frac{v}{3-2v}\Big)\bigg\}\\&~& -2\bigg\{\lambda_1-\frac{v(1-\lambda_2)}{3(1-v)}\bigg\}\sqrt{\frac{v}{3(1-v)}}\tan^{-1}\bigg(\sqrt{\frac{3(1-v)}{v}}\bigg)\\&~&+\frac{2}{9(1-\lambda_2)}\big(\lambda^2_1+13\lambda^2_2-2\lambda_1\lambda_2-18\lambda_2+6\big),
	\end{eqnarray*}
where  $\xi(x)$=$\frac{1}{2r}\big\{2v(\lambda_2-\lambda_1)x+6v(1-\lambda_2)x^2+2(\lambda_2-\lambda_1)(1-v)x^3+6(1-v)(1-\lambda_2)x^4\big\}\log\Big(\frac{x}{x-1}\Big)$, $r=\sqrt{\lambda_1^2+\lambda_2^2+\lambda_1 \lambda_2-3\lambda_1},$ and  the expression of $\phi(x)$ is as provided in Remark \ref{re2.2}.
\end{example}

Using the values of  $KL(f_{mix},f_{X_{CT}})$ and $KL(f_{X_{CT}},f_{mix})$, the Jeffreys' divergence between a mixture distribution and CT distribution can be easily computed.
	\subsection{Cubic transmuted Shannon entropy}
			 We now propose an extension of the Shannon entropy of CT distribution as well as JS divergence, and call it as CT Shannon entropy. The definition of the CT Shannon entropy is given below. Here, we need the PDF of $X_{CT}$, which is given by
			\begin{eqnarray}\label{eq2.8*}	
				f_{X_{CT}}(x)=\lambda_1f(x)+\frac{1}{3}(\lambda_2-\lambda_1)f_{2:3}(x)+\frac{1}{3}(3-\lambda_2-2\lambda_1)f_{\max}(x),
			\end{eqnarray}
			where $\lambda_1\in[0,1]$ and $\lambda_2\in[-1,1)$. The CDF of $X_{CT}$ is denoted by $F_{X_{CT}}(x)$, which can be easily obtained from (\ref{eq2.8*}). Equation (\ref{eq2.8*}) can be obtained from (\ref{(eq1.6)}) by using $F(\cdot),$ $F_{2:3}(\cdot)$ and $F_{\max}(\cdot)$.
			
			 \begin{definition}\label{def2.1}
			 Let $X$ be a random variable with PDF $f(\cdot)$ and $X_{CT}$ be a cubic transmuted random variable with PDF $f_{X_{CT}}(\cdot)$. Then, the CT Shannon entropy between  the PDF $f_{X_{CT}}(\cdot)$ in (\ref{eq2.8*}) and the PDFs of its components $f(\cdot)$, $f_{2:3}(\cdot)$ and $f_{\max}(\cdot)$ is defined as
			 	\begin{equation}\label{eq2.8}
			CTS(f,f_{2:3},f_{\max};\Lambda)= H(f_{X_{CT}})-\lambda_1H(f)-\frac{1}{3}(\lambda_2-\lambda_1) H(f_{2:3})-\frac{1}{3}(3-\lambda_2-2\lambda_1) H(f_{\max}),
			  \end{equation}
			   where $\Lambda=(\lambda_1,\lambda_2)$, $\lambda_1\in[0,1],$ and $\lambda_2\in[-1,1)$.
			 \end{definition}		  		
			  		
	 Now, we establish  an important theorem regarding CT Shannon entropy which shows that  $CTS(f,f_{2:3},f_{\max};\Lambda)$, defined in Definition \ref{def2.1}, can be expressed in terms of the Kullback-Leibler divergence measures between $f_{\max}(\cdot)$ and $f(\cdot),$ $f_{2:3}(\cdot)$ and $f_{\max}(\cdot)$, and $f_{X_{CT}}(\cdot)$ and $f(\cdot)$.
	
	  \begin{theorem}
	  	The CT Shannon entropy can be expressed as
	  	\begin{eqnarray}\label{eq2.9}
	  	\nonumber CTS(f,f_{2:3},f_{\max};\Lambda)&=&(1-\lambda_2)KL(f_{\max},f)+\frac{1}{3}(\lambda_2-\lambda_1)KL(f_{2:3},f_{\max})\\&~&-KL(f_{X_{CT}},f)+C^*
	  	\end{eqnarray}
	  	where $\Lambda=(\lambda_1,\lambda_2)$ and $C^*=\big\{\log (3)-1\big\}(\lambda_2-\lambda_1)$.
	  \end{theorem}	
	
	\begin{proof}
		Substituting $H(f_{X_{CT}})$ from Theorem \ref{th2.1} into (\ref{eq2.8}), we obtain
		\begin{eqnarray}\label{eq2.10}
		 CTS(f,f_{2:3},f_{\max};\Lambda)&=&\frac{2}{3}(\lambda_1-\lambda_2)H(f_{\max})-\frac{1}{3}(\lambda_2-\lambda_1)H(f_{2:3})+H(f_{U_{CT}})\nonumber\\&~&-(1-\lambda_2)H(f_W)+2(\lambda_2-\lambda_1)H^F(f)\nonumber\\
		&=& \frac{1}{3}(\lambda_1-\lambda_2)\left(2H(f_{\max})+H(f_{2:3})-6H^F(f)\right)+H(f_{U_{CT}})\nonumber\\&~&-(1-\lambda_2)H(f_W)\nonumber\\
		&=&-\frac{1}{3}(\lambda_1-\lambda_2)\big\{KL(f_{2:3},f_{\max})+3(\log 3 -1)\big\}+H(f_{U_{CT}})\nonumber\\&~&-(1-\lambda_2)H(f_W).
		\end{eqnarray}
		Now, using the following relations (see Theorem \ref{th2.1*} and Theorem \ref{th2.3})
		$$KL(f_{X_{CT}},f) = KL(f_{U_{CT}},f_U)=-H(f_{U_{CT}})	~\mbox{and}~
		KL(f_{\max},f)= KL(f_W,f_U)=-H(f_{W})$$
		in (\ref{eq2.10}), the desired result follows. Hence, the theorem. 
	  \end{proof}
	
	  Further, bounds of CT Shannon entropy can be found easily in terms of the Kullback-Leibler divergence measures between various components, as follows:
	  \begin{eqnarray}
	  CTS(f,f_{2:3},f_{\max};\Lambda)
	  \left\{
	  \begin{array}{ll}
	  \ge (1-\lambda_2)KL(f_{\max},f)+\frac{1}{3}(\lambda_2-\lambda_1)KL(f_{2:3},f_{\max})\\-KL(f_{X_{CT}},f), ~~~  \mbox{if}~ \lambda_2 \ge \lambda_1,\\
	  \le (1-\lambda_2)KL(f_{\max},f)+\frac{1}{3}(\lambda_2-\lambda_1)KL(f_{2:3},f_{\max})\\-KL(f_{X_{CT}},f) ,~~~ \mbox{if}~ \lambda_2 < \lambda_1.
	  \end{array}
	  \right.
	  \end{eqnarray}
	   A closed-form expression of $CTS(f,f_{2:3},f_{max};\Lambda)$ can be easily obtained since
	  \begin{eqnarray}\label{eq2.12}
	  \nonumber
	    		CTS(f,f_{2:3},f_{\max};\Lambda)
	    		&=&(1-\lambda_2)KL(f_{W},f_{U})+\frac{1}{3}(\lambda_2-\lambda_1)KL(f_{U_{2:3}},f_{W})\\
	    		&~&-KL(f_{U_{CT}},f_U)-\big\{\log (3)-1\big\}(\lambda_1-\lambda_2),
	    	\end{eqnarray}	   	
where $KL(f_{W},f_{U})$, $KL(f_{U_{2:3}},f_{W})$ and $KL(f_{U_{CT}},f_U)$ are as given above. From (\ref{eq2.12}), it is not hard to notice that the CT Shannon entropy is also free from the underlying distribution. 	

\begin{remark}
  We observe that the Jensen-Shannon (JS) entropy for $f(\cdot)$, $f_{2:3}(\cdot)$, and $f_{\max}(\cdot)$ is a special case of CT Shannon entropy when $\lambda_2\geq\lambda_1$ and $\lambda_1,\lambda_2\in[0,1).$ Denote $\delta_1=\lambda_1$, $\delta_2=\frac{1}{3}(\lambda_2-\lambda_1)$ and  $\delta_3=\frac{1}{3}(3-\lambda_2-2\lambda_1)$. Then, evidently,
  	$$ CTS(f,f_{2:3},f_{\max};\lambda_1,\lambda_2)=JS(f,f_{2:3},f_{\max};\delta),$$
  	where $JS(f,f_{2:3},f_{\max};\delta)=H(\delta_1f+\delta_2f_{2:3}+\delta_3f_{\max})-\{\delta_1H(f)+\delta_2H(f_{2:3})+\delta_3H(f_{\max})\}$ with $\delta=(\delta_1,\delta_2,\delta_3)$.
  \end{remark}	
  

\begin{remark}~~
		 	 For $\lambda_1=1+\lambda$ and $\lambda_2=1-\lambda$, (\ref{eq2.9}) reduces to
		 \begin{eqnarray*}
		 		 CTS(f,f_{2:3},f_{\max};\lambda)&=&\lambda KL(f_{\max},f)-\frac{2}{3}\lambda KL(f_{2:3},f_{\max})-KL(f^*_{X_{CT}},f)\\
		 		 &~&+2\lambda\big\{1-\log(3)\big\}, ~|\lambda|\le 1,
		 		\end{eqnarray*}
		 		which is the CT Shannon entropy of the one-parameter CT distribution.
		 	\end{remark}


	\section{Gini information}	 	
	 In this section, we discuss another important information measure, namely,  Gini's mean difference of CT distribution, CT Gini's mean difference and investigate their connection to energy distances.
			
			 \subsection{Gini's mean difference}
			  Let $X$ be an absolutely continuous random variable with CDF $F(\cdot)$. Then, the Gini's mean difference $(GMD)$ of $X$ is given by 
			 		 \begin{eqnarray} \label{eq3.1}
			 		   GMD(F)=2\int_{-\infty}^{\infty}F(x)\bar{F}(x)dx,
			 		   \end{eqnarray}
			 		   where $\bar{F}(x)=1-F(x)$ is the survival function of $X$.
			 		   We now show that  Gini's mean difference for the CT distribution can be written in terms of  GMDs of its components.	 	

  \begin{theorem}\label{th3.1}
		 Let $X$ and $X_{CT}$ be a continuous random variable and its CT random variable, respectively. Then, the Gini's mean difference of $F_{X_{CT}}(\cdot)$ in (\ref{(eq1.6)}) is given by
		 \begin{eqnarray}\label{eq3.2}
		 \nonumber
		  GMD(F_{X_{CT}})&=&\lambda^2_1GMD(F)+(1-\lambda_2)^2GMD(F_{\max})\\
		  		  	&~&+2\int_{0}^{1}\bigg\{\dfrac{Au+Bu^2+Cu^3+Du^4+Eu^5}{f(F^{-1}(u))}\bigg\}du,
		 	\end{eqnarray}	
		   where $A=\lambda_1(1-\lambda_1)$, $B=\lambda_2-\lambda_1$, $C=\lambda_2(1-\lambda_2)-2\lambda_1(\lambda_2-\lambda_1)$, $D=-\big\{(\lambda_2-\lambda_1)^2+2\lambda_1(1-\lambda_2)\big\}$ and $E=-2(1-\lambda_2)(\lambda_2-\lambda_1)$.
		  \end{theorem}
	\begin{proof}
		From (\ref{eq3.1}), we have
		\begin{eqnarray*}
		\frac{1}{2} GMD(F_{X_{CT}})&=&\int_{-\infty}^{\infty}F_{X_{CT}}(x)\bar F_{X_{CT}}(x)dx\\
			&=&\lambda^2_1\int_{-\infty}^{\infty} F(x)\bar F(x)dx+(1-\lambda_2)^2\int_{-\infty}^{\infty} F^3(x)\bar F^3(x)dx\nonumber\\
		  	&~&+\int_{-\infty}^{\infty} \big\{AF(x)+BF^2(x)+CF^3(x)+DF^4+EF^5(x)\big\}dx\nonumber\\
		  	&=&\frac{1}{2}\lambda^2_1GMD(F)+\frac{1}{2}(1-\lambda_2)^2GMD(F_{\max})\nonumber\\
		  	&~&+\int_{0}^{1}\bigg\{\dfrac{Au+Bu^2+Cu^3+Du^4+Eu^5}{f(F^{-1}(u))}\bigg\}du,~~\big(using~ u=F(x)\big)
		\end{eqnarray*}
		which completes the proof of the theorem.
		\end{proof}
In the following, we present upper and lower bounds for the GMD as
		\begin{eqnarray}
		GMD(F_{X_{CT}})
		\left\{
		\begin{array}{ll}
		\ge \lambda^2_1GMD(F)+(1-\lambda_2)^2 GMD(F_{\max}),  & \mbox{if}~ R^*\ge 0\\
		\le \lambda^2_1GMD(F)+(1-\lambda_2)^2 GMD(F_{\max}),  & \mbox{if}~ R^*\le 0,
		\end{array}
		\right.
		\end{eqnarray}
		where $R^*=2\int_{0}^{1}\{(Au+Bu^2+Cu^3+Du^4+Eu^5)/f(F^{-1}(u))\}du$, with $A$, $B$, $C$, $D$, and $E$ being as in Theorem \ref{th3.1}.

		Substituting $\lambda_1=1+\lambda$ and $\lambda_2=1$ in $R^*$, we obtain
		\begin{eqnarray}\label{eq3.3}
       \hat R^*&=&2\lambda^2\int_{-\infty}^{\infty}F^2(x)\bar F^2(x)dx-\lambda(1+\lambda)\int_{0}^{1}\frac{u(1-u)(1+2u)}{f(F^{-1}(u))}du.
		\end{eqnarray}

	\begin{remark}~~
			\begin{itemize}
		\item For $\lambda_1=1+\lambda$ and $\lambda_2=1$ and using (\ref{eq3.3}) in (\ref{eq3.2}), Gini's mean difference of CT distribution becomes
		\begin{eqnarray*}\label{eq3.4}
		GMD(F_{X_{T}})&=&(1+\lambda)^2GMD(F)+\lambda^2GMD(F_{\max})\\
		&~&-\lambda(1+\lambda)\int_{0}^{1}\frac{u(1-u)(1+2u)}{f(F^{-1}(u))}du,
		\end{eqnarray*}	
	which is the Gini's mean difference of transmuted distribution. For the details, see \cite{kharazmi2021informational};
		
		\item For $\lambda_1=1+\lambda$ and $\lambda_2=1-\lambda$, from (\ref{eq3.2}), we have
		\begin{eqnarray*}
			GMD(F^*_{X_{CT}})&=&(1+\lambda)^2GMD(F)+\lambda^2GMD(F_{\max})\\
			&~&-2\lambda\int_{0}^{1}\bigg\{\dfrac{(1+\lambda)u+2u^2-(5+3\lambda^2)u^3+(6\lambda+2) u^4-4\lambda u^5}{f(F^{-1}(u))}\bigg\}du,
		\end{eqnarray*}
		which is the Gini's mean difference of the one-parameter CT distribution.\\
\end{itemize}
		\end{remark}

Next, we present an example to demonstrate the computation of the GMD of a CT distributed random variable with power distribution as the baseline distribution.
		\begin{example}\label{ex3.1}
		Let $X$ be a random variable following power distribution with $CDF$ $F(x)=(\frac{x}{b})^c,~ \mbox{for}~ 0<x<b,~ c>0$. Then, the GMD of the associated CT distribution is given by
			\begin{eqnarray*}
			GMD(F_{X_{TC}})&=&\frac{2bc}{(c+1)(2c+1)}\lambda^2_1+2b(1-\lambda_2)^2\bigg(\frac{1}{3c+1}-\frac{3}{4c+1}+\frac{3}{5c+1}\\
			&~&-\frac{1}{6c+1}\bigg)+2b\bigg(\frac{A}{c+1}+\frac{B}{2c+1}+\frac{C}{3c+1}+\frac{D}{4c+1}+\frac{E}{5c+1}\bigg).
			\end{eqnarray*}
			where  $A$, $B$, $C$, $D$, and $E$ are as given in Theorem \ref{th3.1}.
			For $b=2$ and $c=3$, the term $R^*=2b\bigg(\frac{A}{c+1}+\frac{B}{2c+1}+\frac{C}{3c+1}+\frac{D}{4c+1}+\frac{E}{5c+1}\bigg)$ has been plotted in Figure \ref{fig2} and we notice from it that the term $R^*$ takes positive as well as negative values for some  $(\lambda_1,\lambda_2)$. Moreover,
			
		\begin{eqnarray}
							R^*(\lambda_1,\lambda_2)
							\left\{
							\begin{array}{ll}
							\ge 0,  & \mbox{if}~ \lambda_1 \in[0,0.8] ~~and~~ \lambda_2\in[0.2,1]\\
							\le 0,  & \mbox{if}~  \lambda_1 \in[0.8,1] ~~and~~ \lambda_2\in[-1,0.2].
							\end{array}
							\right.
							\end{eqnarray}
				
				  	\begin{figure}[htbp!]
				  	    	\centering
				  	    	\includegraphics[width=15cm,height=10cm]{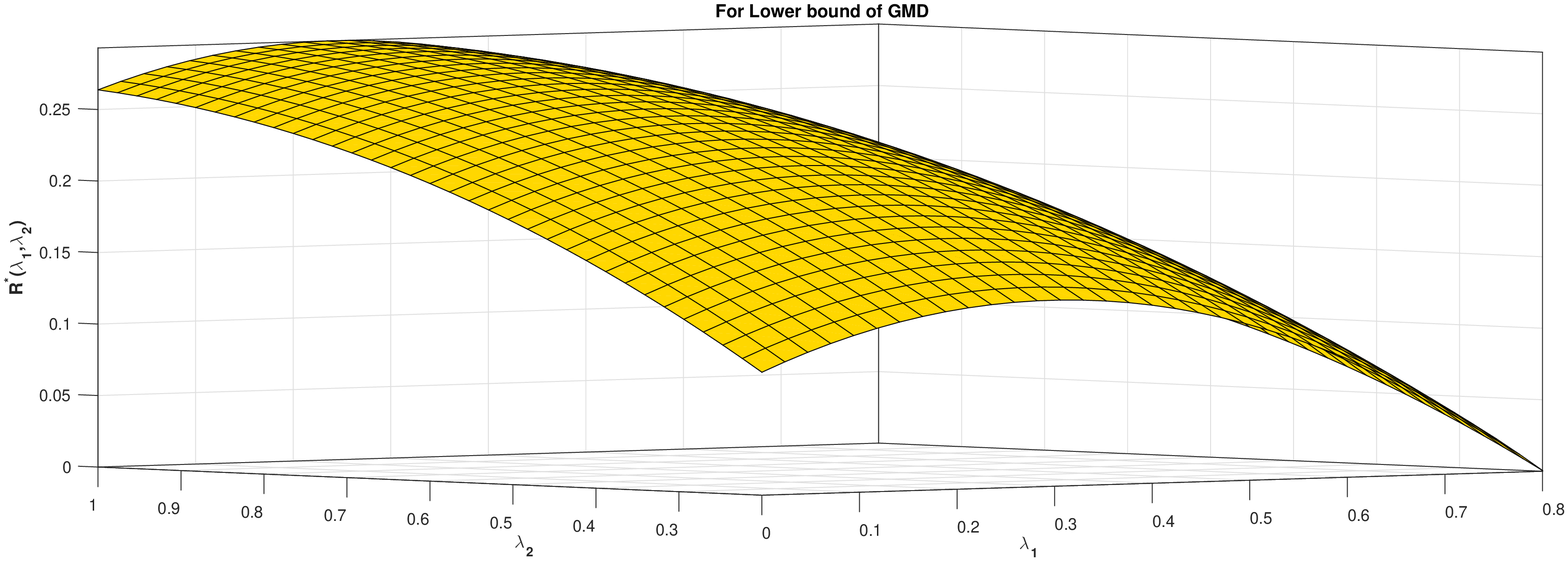}
				  	    	\includegraphics[width=15cm,height=10cm]{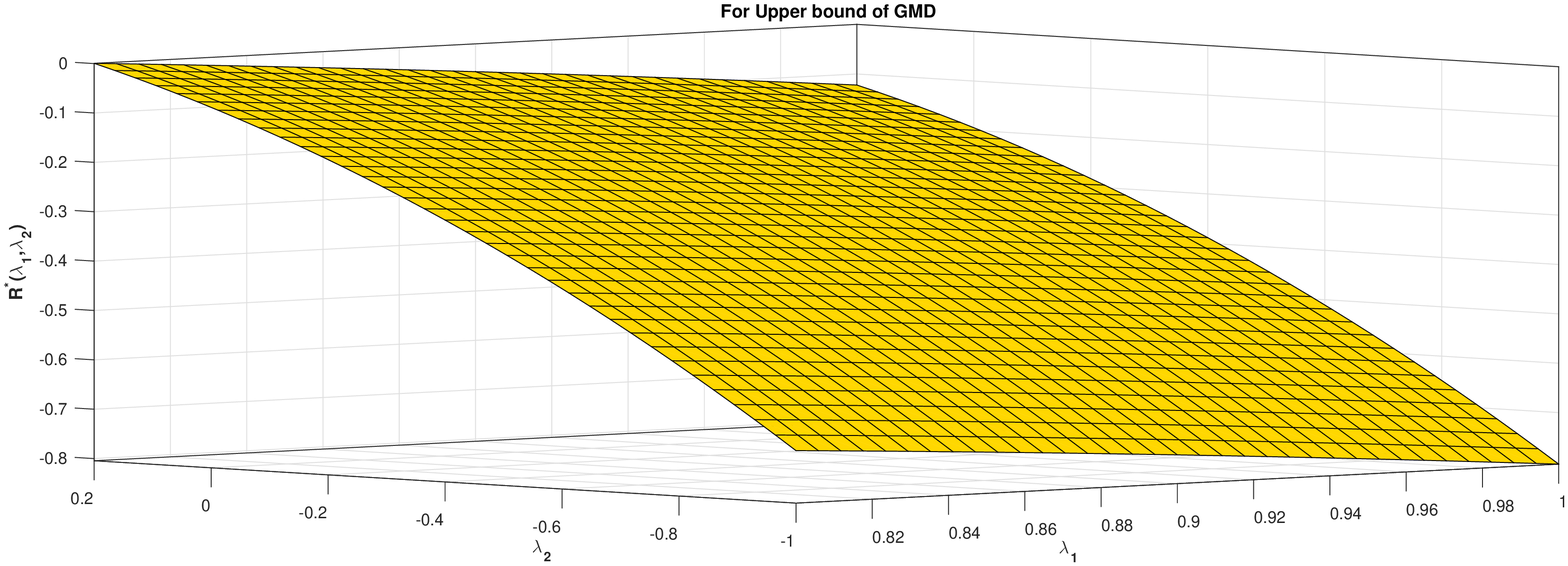}
				  	    	\caption{Plots of $R^*(\lambda_1,\lambda_2)$  in Example \ref{ex3.1}.}
				  	    	\label{fig2}
				  	    \end{figure}

		\end{example}
	\subsection{Cubic Transmuted Gini's mean difference}
		In this subsection, we propose an extension of Gini's mean difference for the cubic transmuted distribution as well as Jensen-Gini divergence. This extended version of the measure is called  cubic transmuted Gini's mean difference and is as defined below.
		\begin{definition}
		Suppose $X$ and $X_{CT}$ have CDFs $F(\cdot)$ and $F_{X_{CT}}(\cdot)$, respectively. Then, the cubic transmuted Gini's mean difference $(CTG)$ between $F_{X_{CT}}(\cdot)$ and its components' CDFs $F(\cdot)$, $F_{2:3}(\cdot)$, and $F_{\max}(\cdot)$ is given by
		\begin{eqnarray}\label{eq3.5*}
				 	\nonumber
				CTG(F,F_{2:3},F_{\max};\Lambda)&=& GMD(F_{X_{CT}})-\lambda_1GMD(F)-\frac{1}{3}(\lambda_2-\lambda_1) GMD(F_{2:3})\\
				   &~&-\frac{1}{3}(3-\lambda_2-2\lambda_1) GMD(F_{\max}),
				  \end{eqnarray}
				   where $\Lambda=(\lambda_1,\lambda_2)$, $\lambda_1\in[0,1],$ and $\lambda_2\in[-1,1)$.
		\end{definition}
		We now derive a connection of CT Gini's mean difference with energy distance in the following Theorem.
		  	Let $X_1$ and $X_2$ be two random variables with CDFs $F_1(\cdot)$ and $F_2(\cdot),$ respectively. Then, the energy distance between  $F_1(\cdot)$ and $F_2(\cdot)$ is given by
		  	\begin{eqnarray}\label{eq3.5}
		  		CD(F_1,F_2)=\int_{-\infty}^{\infty}\big\{F_1(x)-F_2(x)\big\}^2dx.
		  	\end{eqnarray}
		\begin{theorem}\label{th3.2}
		The CT Gini's mean difference can be presented through energy distances between $F_{X_{CT}}(\cdot)$ with $F(\cdot)$, $F_{2:3}(\cdot)$, and $F_{\max}(\cdot)$ as follows:
	\begin{eqnarray*}
				CTG(F,F_{2:3},F_{\max};\Lambda)&=&\eta_1~CD(F,F_{X_{CT}})
				+\eta_2~CD(F_{\max},F_{X_{CT}})\\
				&~&+(1-\eta_1-\eta_2)~CD(F_{2:3},F_{X_{CT}}),
	\end{eqnarray*}
	where
	 	\begin{eqnarray*}
	 	\Lambda&=&(\lambda_1,\lambda_2),~~
	\eta_1=\frac{1}{1-2\lambda_1}(4\lambda_1-2\lambda_2-3\lambda^2_1), \\
	\eta_2&=&\frac{1}{9(2\lambda_2-3)}\big\{(26\lambda_2-20\lambda_1-6\lambda^2_2)-3\eta_1(4\lambda_2-8)\big\}.
		\end{eqnarray*}
		\end{theorem}			
	\begin{proof}
		Using (\ref{eq3.2}) in (\ref{eq3.5*}), we have
	\begin{eqnarray*}
			CTG(F,F_{2:3},F_{\max};\Lambda)&=&(\lambda^2_1-\lambda_1)GMD(F(x))-\frac{1}{3}(3\lambda^2_2-5\lambda_2+2\lambda_1)GMD(F_{\max}(x))\\
			&~& +2\int_{-\infty}^{\infty}\big\{AF(x)+BF^2(x)+CF^3(x)+DF^4(x)+EF^5(x)\big\}dx\\
			&~& -\frac{1}{3}(\lambda_2-\lambda_1)GMD(F_{2:3}(x)).
	\end{eqnarray*}
	Further, using (\ref{eq3.1}) in the above equation and after some algebraic calculations, we obtain
	\begin{eqnarray*}
	 CTG(F,F_{2:3},F_{\max};\Lambda)&=&\int_{-\infty}^{\infty}\big\{(4\lambda_1-2\lambda^2_1-2\lambda)F^2(x)+\frac{4}{3}(3\lambda^2_1-3\lambda_1\lambda_2+\lambda_2-\lambda_1)
	F^3(x)\\&~&+(12\lambda_2-16\lambda_1+8\lambda_1\lambda_2-2\lambda^2_1-2\lambda^2_2)F^4(x)\\
	&~&+4(\lambda^2_2-\lambda_2\lambda_1+5\lambda_1-5\lambda_2)F^5(x)
	+\frac{2}{3}(\lambda_2-10\lambda_1-3\lambda^2_2)F^6(x)\big\}dx\\
&=&\eta_1~CD(F,F_{X_{CT}})
				+\eta_2~CD(F_{\max},F_{X_{CT}})\\
				&~&+(1-\eta_1-\eta_2)~CD(F_{2:3},F_{X_{CT}}),
	\end{eqnarray*}
	where the last equality is obtained using (\ref{eq3.5}). Thus, the proof gets completed.
		\end{proof}
	
		In the following remark, we present CT Gini's mean difference for some special case of the CT distributions.

		\begin{remark}~~
			For $\lambda_1=1+\lambda$ and $\lambda_2=1-\lambda$, from Theorem \ref{th3.2}, we get
		\begin{eqnarray*}
			CTG(F,F_{2:3},F_{\max};\lambda)&=&\eta^*_1CD(F,F^*_{X_{CT}})
			+\eta^*_2CD(F_{\max},F^*_{X_{CT}})\\
			&~&+(1-\eta^*_1-\eta^*_2)CD(F_{2:3},F^*_{X_{CT}})
		\end{eqnarray*}
		as the CT Gini's mean difference of the one parameter CT distribution,
		where $\eta^*_1=\frac{1}{1+2\lambda}(3\lambda^2+2)$ and $\eta^*_2=\frac{1}{9(1+2\lambda)}\{(34\lambda+6\lambda^2)+12\eta^*_1(1+\lambda)\}$.
		\end{remark}

\section{Chi-square divergence}
		In this section, we derive Chi-square $(\chi^2)$ divergence between the CT distribution with CDF $F_{X_{CT}}(\cdot)$ given in (\ref{(eq1.6)}) and its components' CDFs  $F(\cdot)$, $F_{2:3}(\cdot)$, and $F_{max}(\cdot)$.
		Suppose $X$ and $Y$ are two random variables with PDFs $f_1(\cdot)$ and $f_2(\cdot)$, respectively. Then, $\chi^2$ divergence (see \cite{nielsen2013chi}) between  $f_1(\cdot)$ and $f_2(\cdot)$ is given by
		\begin{eqnarray} \label{eq4.1}
		\chi^2(f_1,f_2)=\int_{-\infty}^{\infty} \frac{\big\{f_1(x)-f_2(x)\big\}^2}{f_2(x)}dx.
		\end{eqnarray}

        Now, we present some important properties of $\chi^2$ divergence in the following theorem.
\begin{theorem}\label{th4.1}
	Suppose  $X_1$, $X_2$ and $X_3$ are i.i.d. as $X$ with PDF $f(\cdot)$. Then, the Chi-square divergence between $f_{X_{CT}}(\cdot)$ with $f(\cdot)$ and $f_{\max}(\cdot)$ are given by
	\begin{itemize}
	\item[(a)] $\chi^2_{CT}(f_{X_{CT}},f)=\chi^2_{CT}(f_{U_{CT}},f_U),$
	\item[(b)] $\chi^2_{CT}(f,f_{X_{CT}})=\chi^2_{CT}(f_U,f_{U_{CT}} ),$
	\item[(c)] $\chi^2_{CT}(f_{\max},f_{X_{CT}})= \chi^2_{CT}(f_W,f_{U_{CT}}).$
	\end{itemize}
\end{theorem}

\begin{proof}
$(a)$ We have
\begin{eqnarray*}
	\chi^2_{CT}(f_{X_{CT}},f)&=& \int_{-\infty}^{\infty}\dfrac{\big\{f_{X_{CT}}(x)-f(x)\big\}^2}{f(x)}dx\\
	&=&\int_{-\infty}^{\infty}\dfrac{\big[f(x)\big\{\lambda_1+2(\lambda_2-\lambda_1)F(x)+3(1-\lambda_2)F^2(x)\big\}-f(x)\big]^2}{f(x)}dx\\
	&=&\int_{0}^{1}\big[\big\{\lambda_1+2(\lambda_2-\lambda_1)u+3(1-\lambda_2)u^2\big\}-1\big]^2du~~\big(\mbox{using u=F(x)}\big)\\
	&=&\chi^2_{CT}(f_{U_{CT}},f_U).
\end{eqnarray*}
 parts $(b)$ and $(c)$ can be proved similarly.
\end{proof}
 Closed-form expressions of the Chi-square divergences $\chi^2_{CT}(f_{U_{CT}},f_U)$, $\chi^2_{CT}(f_U,f_{U_{CT}}),$ and $\chi^2_{CT}(f_W,f_{U_{CT}})$ are obtained below. The detailed derivations are omitted here for conciseness:
	\begin{eqnarray*}
		(a)~\chi^2_{CT}(f_{U_{CT}},f_U)&=&\frac{1}{15}(12-15\lambda_1-9\lambda_2+5\lambda_1\lambda_2+5\lambda^2_1+2\lambda^2_2),\\
		(b)~\chi^2_{CT}(f_U,f_{U_{CT}})&=&-1+\frac{1}{r}\tanh^{-1}(\lambda),\\
		(c)~\chi^2_{CT}(f_W,f_{U_{CT}})&=&-1+\frac{1}{3(1-\lambda_2)^3}(4\lambda^2_1+10\lambda^2_2-8\lambda_1\lambda_2-9\lambda_2+3)-9\sigma(p,q),
	\end{eqnarray*}
	where
	$\lambda=\frac{r}{\lambda_2}$ and
	$\sigma(p,q)=\frac{1}{2r}\Big\{p^4\log\big(\frac{p}{p-1}\big)-q^4\log\big(\frac{q}{q-1}\big)\Big\}.$


\begin{remark}
	Making use of Theorem \ref{th4.1} and the expressions of $\chi^2_{CT}(f_{U_{CT}},f_U)$, $\chi^2_{CT}(f_U,f_{U_{CT}}),$ and $\chi^2_{CT}(f_W,f_{U_{CT}})$, we observe that  Chi-square divergence measures are free of the underlying distribution function $F(\cdot)$.
\end{remark}

The Symmetric $\chi^2$-divergence between two density functions $f_1(\cdot)$ and $f_2(\cdot)$ is defined as (see \cite{dragomir2000new})
\begin{eqnarray}\label{eq4.2}
	\Psi(f_1,f_2)=\int_{-\infty}^{\infty}\frac{\big\{f_1(x)-f_2(x)\big\}^2\big\{f_1(x)+f_2(x)\big\}}{f_1(x)f_2(x)}dx.
\end{eqnarray}
Note that, while the Chi-square divergence is non-symmetric, but (\ref{eq4.2}) is symmetric, that is, 
	$\Psi(f_1,f_2)=\Psi(f_2,f_1).$
In the following remark, we show that the symmetric Chi-square divergence between $f_{X_{CT}}(\cdot)$ and $f(\cdot)$ is equal to that between $f_{U_{CT}}(\cdot)$ and $f_U(\cdot)$.

\begin{remark}
	We have
	$$\Psi(f_{X_{CT}},f)=\chi^2(f_{U_{CT}},f_U)+\chi^2(f_U,f_{U_{CT}})=\Psi(f_{U_{CT}},f_U).$$
	It is clear that the symmetric Chi-square divergence between $f_{X_{CT}}(\cdot)$ and $f(\cdot)$ is free from the underlying distribution $F(\cdot)$.
\end{remark}

 \section{Fisher information}
	     In this section, we study Fisher information for the CT distribution with CDF in (\ref{(eq1.6)})  and for the one-parameter CT distribution with CDF in (\ref{eq1.8}).
	     \subsection{Fisher information matrix of CT distribution for  $\lambda_1$ and $\lambda_2$}
	    Below, we present the Fisher information matrix for the cubic transmuted distribution.
	     \begin{theorem}
	     Suppose $X_{CT}$ is a cubic transmuted random variable with PDF $f_{CT}(\cdot)$ as in (\ref{eq1.7}). Then, the Fisher information matrix is 
	     $$I(\Lambda)=\begin{bmatrix}
	       	E(\tau\rho_1) & E(\tau\rho_2)\\
	       	E(\tau\rho_2) & E(\tau\rho_3)
	       	\end{bmatrix},$$
	       	where $\tau=\frac{f^2(x)}{f^2_{X_{CT}}}$, $\rho_1\big\{1-2F(x)\big\}^2$, $\rho_2=\big\{2F(x)-7F^2(x)+6F^3(x)\big\}$, $\rho_3=\big\{2F(x)-3F^2(x^2)\big\}^2$, and $\Lambda=(\lambda_1,\lambda_2)$.
	     \end{theorem}
	     \begin{proof}
	     Using log-likelihood function
	     $$l(\Lambda)=\log \Big(f(x)\big\{\lambda_1+2(\lambda_2-\lambda_1)F(x)+3(1-\lambda_2)F^{2}(x)\big\}\Big),$$
	    the Fisher information matrix is easily obtained.
	     \end{proof}

        We now present an example to illustrate the computation of the Fisher information matrix for the CT uniform distribution.

	       \begin{example}
	       	Let the baseline distribution for the CT distribution be uniform  in $(0,1)$. Then, the Fisher information matrix of the CT uniform distribution is given by
	       	\begin{eqnarray*}
	       		I(\Lambda)&=&\begin{bmatrix}
	       			E(\tau\rho_1) & E(\tau\rho_2)\\
	       			E(\tau\rho_2) & E(\tau\rho_3)
	       		\end{bmatrix}\\
	       		&=&\begin{bmatrix}
	       			\frac{1}{r}\tanh^{-1}(\lambda)+\omega_1+\dfrac{4}{3(1-\lambda_2)} & -\omega_2-\dfrac{4(1-\lambda_1)}{3(1-\lambda_2)^2} \\
	       			-\omega_2-\dfrac{4(1-\lambda_1)}{3(1-\lambda_2)^2} & \frac{1}{3(1-\lambda_2)^3}(4\lambda^2_1-8\lambda_1+3\lambda_2+1)-\omega_3
	       		\end{bmatrix},
	       	\end{eqnarray*}
	       \end{example}
	       where $r, ~p, ~q$ are as given in Remark \ref{re2.2}, $\lambda=r/\lambda_2,$ and
	       \begin{eqnarray*}
	       \omega_1&=&\frac{2}{r}\bigg\{(p-p^2)\log \Big(\frac{p}{p-1}\Big)-(q-q^2)\log \Big(\frac{q}{q-1}\Big)\bigg\},\\
	       \omega_2&=&\frac{1}{2r}\bigg\{(2p-7p^2+6p^3)\log \Big(\frac{p}{p-1}\Big)-(2q-7q^2+6q^3)\log \Big(\frac{q}{q-1}\Big)\bigg\},\\
	       \omega_3&=&\frac{1}{2r}\bigg\{(4p^2-12p^3+9p^4)\log \Big(\frac{p}{p-1}\Big)-(4q^2-12q^3+9q^4)\log \Big(\frac{q}{q-1}\Big)\bigg\}.
	       \end{eqnarray*}
	
	       	 	For $\lambda_1=1+\lambda$ and $\lambda_2=1-\lambda$, the CT distribution with CDF $F_{X_{CT}}(\cdot)$ in (\ref{(eq1.6)}) reduces to the one-parameter CT distribution with CDF $F^*_{X_{CT}}(\cdot)$ in (\ref{eq1.8}), which discussed by \cite{al2017cubic}. In the next subsection, we discuss Fisher information for this special case.
	       	\subsection{Fisher information for one-parameter CT distribution}
	       	The Fisher information of a continuous random variable $X$ with PDF $f(x;\theta)$, where $\theta$ is a parameter, is given by (see \cite{fisher1929tests})
	       	\begin{eqnarray*}
	       	\textit{I}(\theta)=\int_{-\infty}^{\infty}\bigg[\frac{\partial\log\big( f(x;\theta)\big)}{\partial\theta}\bigg]^2f(x;\theta)dx.
	       	\end{eqnarray*}
	       	We now establish that the Fisher information for the one-parameter CT distribution can be represented in terms of Chi-square divergence.
	       	\begin{theorem}
	       	The Fisher information for the one-parameter CT distribution with PDF $f^*_{X_{TC}}(\cdot)$ in (\ref{eq1.9}) is given by
	       	 \begin{eqnarray*}
	       	 \textit{I}(\lambda)=\frac{1}{\lambda^2}\chi^2(f,f^*_{X_{TC}}).
	       	\end{eqnarray*}
	       	\end{theorem}
	       	\begin{proof}
	       Using the definition of Fisher information for the parameter $\lambda$, we have
	        \begin{eqnarray*}
	        \textit{I}(\lambda)&=& E\bigg[\frac{\partial}{\partial\lambda}\log \big(f^*_{X_{CT}}(X)\big)\bigg]^2\\
	        &=&E\bigg[\frac{f(X)\big\{1-4F(X)+3F^2(X)\big\}}{f^*_{X_{CT}}(X)}\bigg]^2\\	     &=&\int_{-\infty}^{\infty}\bigg[\frac{f(x)\big\{1-4F(x)+3F^2(x)\big\}}{f^*_{X_{CT}}(x)}\bigg]^2 f^*_{X_{CT}}(x)dx  \\
	        &=&\frac{1}{\lambda^2}\int_{-\infty}^{\infty}\dfrac{\big\{f(x)-f^*_{X_{CT}}(x)\big\}^2}{f^*_{X_{CT}}(x)}dx\\
	        &=&\frac{1}{\lambda^2}\chi^2(f,f^*_{X_{CT}}).
	        \end{eqnarray*}
	         Hence, the desired result.
	        \end{proof}
	       Closed-form expression of $ \textit{I}(\lambda)$ can be easily obtained as
	        \begin{eqnarray*}
	        \textit{I}(\lambda)=\frac{1}{\lambda^2}\chi^2(f_U,f^*_{U_{TC}})=\frac{1}{\lambda^2}\bigg\{\frac{\lambda^2}{r^*}\tanh^{-1}\bigg(\frac{r^*}{1-\lambda}\bigg)+\pi(p^*)-\pi(q^*)-\frac{\lambda}{3}(13\lambda+10)\bigg\},
	         \end{eqnarray*}
	         where $r^*=\sqrt{\lambda^2-3\lambda}$, $p^*=\frac{r^*-2\lambda}{3\lambda}$, $q^*=\frac{r^*+2\lambda}{3\lambda}$, and $\pi(x)=\frac{\lambda^2}{2r^*}(9x^4-24x^3+22x^2-8x)\log\big(\frac{x}{x-1}\big)$.

	         
	        \begin{remark}
	       We observe from Theorem \ref{th4.1} that the Fisher information for the one-parameter CT distribution is free from the underlying distribution $F^*(\cdot)$ and it is non-symmetric, but the Fisher information of transmuted distribution  obtained by \cite{kharazmi2021informational} is symmetric.
	        \end{remark}

        \section{Simulation study} In this section, we carry out a simulation study to demonstrate the importance of  Kullback-Leibler divergence for the model/component selection. The simulation has been conducted by using $R$ software, based on $500$ replications. First, we generate cubic transmuted uniform random variables using the following algorithm.\\
        	\\
       	Algorithm:\\
---------------------------------------------------------------------------------------------------------------
   \begin{itemize}
   \item [Step-1:] Generate $V$ from $U(0,1)$ distribution;
   \item [Step-2:] Generate three independent $U(0,1)$ distributed random variables, say $X_{1},~X_{2},$ and $X_{3};$ 
   \item [Step-3:] Arrange $X_{1},~X_{2},$ and $X_{3}$ in ascending order and denote them by $X_{1:3}<X_{2:3}<X_{3:3};$
   \item [Step-4:] If $V<\pi_{1},$ then $X_{CT}=X_{1:3}$; but, if $\pi_{1}<V<\pi_{2},$ then $X_{CT}=X_{2:3}$; otherwise $X_{CT}=X_{3:3}$;
   \item [Step-5:] To generate $n$  cubic transmuted uniform  random variables, repeat Steps $1$-$4$ $n$ times. 
   \end{itemize}
---------------------------------------------------------------------------------------------------------------\\
        	
     Now, we compute $KL(f_{X_{CT}},f_{1:3})~(=KL1)$, $KL(f_{X_{CT}},f_{2:3})~(=KL2)$, and $KL(f_{X_{CT}},f_{3:3})~(=KL3)$ from the simulated data sets for different choices of  mixing proportions $(\pi_1,\pi_2,\pi_3)$ and sample sizes $(n)$. In addition, we  also compute three model/component proportions based on $KL1$, $KL2$, and $KL3$. All the computed values with difference choices of  mixing proportions and sample sizes are presented in Table \ref{tb1}. From Table \ref{tb1}, we notice the following:
          \begin{itemize}
          	\item [(i)] The values of  KL divergences (here $KL1$, $KL2$, and $KL3$) maintain the order corresponding to the values of  mixing proportions, as one would expect. For example, when $\pi_1=0.9$, $\pi_2=0.05$, and $\pi_3=0.05$, the $KL1$ takes smaller value than $KL2$ and $KL3$, which is clearly observed in Table \ref{tb1}; 
          	
          		\item [(ii)] From Table \ref{tb1}, we also notice that the model/component proportion is largest corresponding to the largest mixing proportion. Further, the model/component proportions get improved when the sample size increases.  
          \end{itemize}
      Based on the above observations, we can conclude that the Kullback-Leibler divergence can be quite useful  as a model selection criterion.      	
        	
  \begin{table}[ht]
     	\centering 
     	\caption {The KL divergences and model/component proportions for different choices of mixing proportions and sample sizes.}
     	\scalebox{0.6}{\begin{tabular}{c c c c c c c c } 
     			\hline\hline\vspace{.25cm} 
     			$\textbf{Mixing proportions}$ &$ \bf{n} $& $\bf{KL1}$ &$\textbf{KL2}$ &$\textbf{KL3}$ &$\textbf{Proportion~(KL1)}$ & $\textbf{Proportion~(KL2)}$ & $\textbf{Proportion~(KL3)}$ \\
     			\hline\hline
     			
     			$(0.9,0.05,0.05)$ & 150 & 0.1461 &  0.2550 & 0.8350 & 0.954 & 0.046 & 0.000 \\[1.2ex]
     			~ & 300 & 0.1593 &  0.2548 & 0.8874 & 0.990 & 0.010 & 0.000 \\[1.2ex]
     			~ & 500 &  0.0671 &  0.2625 & 0.8963& 1.000 & 0.000 & 0.000 \\[1.2ex]
     			\hline
     			$(0.05,0.9,0.05)$ & 150 & 0.0917 &  0.0554 & 0.1712 & 0.236 & 0.764 & 0.000 \\[1.2ex]
     			~ & 300 &  0.0886&  0.0459 &0.1483 & 0.050 & 0.950 & 0.000 \\[1.2ex]
     			~ & 500 & 0.3327 &  0.0497& 0.1598 & 0.000 & 1.000 & 0.000 \\[1.2ex]
     			\hline
     			$(0.05,0.05,0.9)$ & 150 & 0.8201&  0.2103 & 0.0836 & 0.000 & 0.002 & 0.998 \\[1.2ex]
     			~ & 300 &  0.829 &  0.3062 & 0.1483 & 0.000 & 0.000 & 1.000 \\[1.2ex]
     			~ & 500 & 1.341 & 0.3432 & 0.1045 & 0.000 & 0.000 & 1.000 \\[1.2ex]			        
     			\hline
     			$(0.9,0.075,0.025)$ & 150 & 0.1461&  0.2550 & 0.8350 & 0.954 & 0.046 & 0.000 \\[1.2ex]
     			~ & 300 & 0.1593 &  0.2548& 0.8874 & 0.990 & 0.010 & 0.000 \\[1.2ex]
     			~ & 500 & 0.0671& 0.2625 & 0.8963 & 1.000 & 0.000 & 0.000 \\[1.2ex]	
     			
     			\hline		
     			$(0.075,0.9,0.025)$ & 150 & 0.0758&  0.0586 &  0.1748 & 0.424 & 0.576 & 0.000 \\[1.2ex]
     			~ & 300 & 0.0719 & 0.0500 & 0.1532 & 0.242 & 0.758 & 0.000 \\[1.2ex]
     			~ & 500 & 0.3102 & 0.0537 & 0.1648 & 0.000 & 1.000 & 0.000 \\[1.2ex]
     			\hline		
     			$(0.1,0.8,0.1)$ & 250 & 0.0804 &  0.0571 &  0.1009 & 0.214 & 0.746 & 0.040 \\[1.2ex]
     			~ & 300  & 0.0857 &  0.0599 &   0.0938 & 0.154 & 0.778 & 0.068 \\[1.2ex]
     			~ & 500 &  0.3427 & 0.0606& 0.1003 & 0.000 & 0.994 & 0.006 \\	[1.2ex]	
     			\hline
     			$(0.1,0.1,0.8)$ & 150 &  0.6503 &   0.1225 &   0.1031 & 0.000 & 0.338 & 0.662 \\[1.2ex]
     			~ & 300  &  0.6451 & 0.197 &  0.1125 & 0.000 & 0.000 & 1.000 \\[1.2ex]
     			~ & 500 &   1.1142 & 0.2300& 0.1171& 0.000 & 0.000 & 1.000 \\[1.2ex]			
     			\hline		
     			$(0.3,0.6,0.1)$ & 30 &  0.1841&   0.1628 &  0.2229 & 0.362 & 0.638 & 0.000 \\[1.2ex]
     			~ & 40 &   0.1440 & 0.0953& 0.2459 & 0.224 & 0.776 & 0.000 \\[1.2ex]			
     			~ & 50  &  0.1814 & 0.1005 &   0.2623 & 0.030 & 0.970 & 0.000 \\[1.2ex]
     			
     			\hline\hline	 		
     	\end{tabular}} 
     
     	\label{tb1} 
     \end{table}        	
        	
 Further, to obtain  $100\times(1-\gamma)\%$ confidence intervals for the unknown model parameters of the cubic transmuted Weibull (CTW) distribution, we used the Fisher information matrix, provided in Section $5$. The probability density function of the CTW distribution is given by
 \begin{align}
 f_{X_{CT}}(x)=kx^{k-1}e^{-x^k}\{\lambda_1+2(\lambda_2-\lambda_1)(1-e^{-x^k})+3(1-\lambda_2)(1-e^{-x^k})^2\},
 \end{align}
 where $x>0,~k>0,~ \lambda_1 \in[0,1]$ and $\lambda_2 \in[-1,1)$. Using the inverse Fisher information matrix as the variance-covariance matrix computed at MLEs of parameters, $100\times(1-\gamma)\%$ approximate confidence intervals of unknown model parameters were obtained and are provided in Table \ref{tb2}. Here, we obtained $90\%$ and $95\%$ confidence intervals for different choices of the parameters and sample sizes $(n)$. It is seen from Table \ref{tb2} that as $n$ increases, the average width of the intervals decrease. In addition, we have also considered cubic transmuted uniform distribution with probability density function, given by
 \begin{align}
 f_{X_{CT}}(x)=\lambda_1+2(\lambda_2-\lambda_1)x+3(1-\lambda_2)x^2,~x\in[0,1],~\lambda_1 \in[0,1],~ \lambda_2 \in[-1,1),
 \end{align}
 and obtained the approximate confidence intervals for the model parameters, and these are presented in Table \ref{tb3}. From its, we notice similar observations as in the case of  CTW distribution.

   \begin{table}[ht]
   	\centering 
   		\caption { $90\%$ and $95\%$ confidence intervals for the model parameters of CTW distribution.}
   	\scalebox{0.8}{\begin{tabular}{c c c c c c c c } 
   			\hline\hline\vspace{.25cm} 
   			$\textbf{Parameters}$  &$\bf{n}$    &$\textbf{90\% confidence interval}$ &$\textbf{95\% confidence interval}$ \\
   			\hline\hline
   			
   			$\lambda_1=0.4$ & 150  &  $(0.3852,~0.5874)$ & $(0.3504,~0.6221)$  \\[1.2ex]
   			~ & 300  &  $(0.3738,~0.5360)$ & $(0.346,~0.5638)$  \\[1.2ex]
   			~ & 500  &  $(0.3318,~0.4489)$ & $(0.3117,~0.469)$  \\[1.2ex]
   			\hline
   			$\lambda_2=0.6$ & 150  &  $(0.1245,~0.5328)$ & $(0.0544,~0.6029)$  \\[1.2ex]
   			~ & 300 &  $( 0.2825,~0.5885)$ & $(0.2299,~0.6410)$   \\[1.2ex]
   			~& 500  &  $(0.4679,~0.6971)$ & $(0.4285,~0.7364)$  \\[1.2ex]

   			\hline
   			$k=1.0$ & 150  &  $(1.0018,~1.0675)$ & $(0.9906,~1.0788)$  \\[1.2ex]
   			~ & 300  &  $(0.9900,~1.0389)$ & $(0.9816,~1.0473)$   \\[1.2ex]
   			~ & 500  &  $(1.0082,~1.0465)$ & $(1.0016,~1.0531)$  \\[1.2ex]			  
   			\hline

   			$\lambda_1=0.7$ & 150  &  $( 0.5658,~0.7978)$ & $(0.526,~0.8376)$  \\[1.2ex]
   			~ & 300  &  $(0.7307,~0.928)$ & $(0.6968,~0.9618)$  \\[1.2ex]
   			~ & 500  &  $(0.6455,~ 0.7963)$ & $(0.6196,~0.8222)$  \\[1.2ex]
   			\hline
   			$\lambda_2=0.5$ & 150  &  $(0.2113,~0.6495)$ & $(0.1361,~0.7247)$  \\[1.2ex]
   			~ & 300 &  $( 0.0099,~0.3560)$ & $(0.0000,~0.4154)$   \\[1.2ex]
   			~& 500  &  $(0.3027,~ 0.5696)$ & $(0.2569,~0.6154)$  \\[1.2ex]		  
   			\hline
   			$k=1.5$ & 150  &  $(1.4705,~1.5756)$ & $(1.4524,~1.5937)$  \\[1.2ex]
   			~ & 300  &  $(1.5111,~1.5893)$ & $(1.4977,~1.6027)$   \\[1.2ex]
   			~ & 500  &  $(1.5218,~1.5854)$ & $(1.5108~,1.5963)$  \\[1.2ex]	
   			\hline
   			$\lambda_1=0.8$ & 150  &  $( 0.5991,~0.8646)$ & $(0.5535,~0.9102)$  \\[1.2ex]
   			~ & 300  &  $(0.856,~1.0741)$ & $( 0.8186,~1.1116)$  \\[1.2ex]
   			~ & 500  &  $(0.7852,~ 0.964)$ & $( 0.7545,~0.9947)$  \\[1.2ex]
   			\hline
   			$\lambda_2=0.8$ & 150  &  $(0.6285,~1.1027)$ & $(0.5471,~1.1841)$  \\[1.2ex]
   			~ & 300 &  $( 0.2188,~0.5915)$ & $(0.1548~,0.6555)$   \\[1.2ex]
   			~& 500  &  $(0.4928,~ 0.7981)$ & $(0.4404,~0.8505)$  \\[1.2ex]		  
   			\hline
   			$k=2.0$ & 150  &  $(1.9274,~2.0909)$ & $(1.8993,~ 2.1189)$  \\[1.2ex]
   			~ & 300  &  $(2.0405,~ 2.1589)$ & $(2.0202,~2.1792)$   \\[1.2ex]
   			~ & 500  &  $(2.0535,~ 2.1543)$ & $( 2.0362~,2.1716)$  \\[1.2ex]

   			\hline\hline	 		
   	\end{tabular}} 
   
   	\label{tb2} 
   \end{table}

   \begin{table}[ht]
  \centering 
   	 \caption { $90\%$ and $95\%$ confidence intervals of the parameters of cubic transmuted uniform distribution.}	
   	\scalebox{0.8}{\begin{tabular}{c c c c c c c c } 
   			\hline\hline\vspace{.25cm} 
   			$\textbf{Parameters}$  &$\bf{n}$    &$\textbf{90\% confidance interval}$ &$\textbf{95\% confidance interval}$ \\
   			\hline\hline
   			
   			$\lambda_1=0.4$ & 150  &  $(0.3483,~0.5198)$ & $(0.3188,~ 0.5492)$  \\[1.2ex]
   			~ & 300  &  $(0.3625,~0.4931)$ & $(0.34,~0.5155)$  \\[1.2ex]
   			~ & 500  &  $(0.2973,~0.3909)$ & $(0.2812,~0.407)$  \\[1.2ex]
   			\hline
   			$\lambda_2=0.6$ & 150  &  $(0.2539,~0.6138)$ & $(0.1921,~0.6756)$  \\[1.2ex]
   			~ & 300 &  $( 0.3588,~0.6131)$ & $(0.3151,~0.6567)$   \\[1.2ex]
   			~& 500  &  $(0.5729,~0.7660)$ & $(0.5398,~0.7992)$  \\[1.2ex]	
   			
    			\hline
   			  			
   			$\lambda_1=0.7$ & 150  &  $( 0.5575,~0.752)$ & $(0.5241,~0.7854)$  \\[1.2ex]
   			~ & 300  &  $(0.6742,~0.828)$ & $(0.6478,~0.8543)$  \\[1.2ex]
   			~ & 500  &  $(0.5836,~ 0.6978)$ & $(0.5639,~0.7174)$  \\[1.2ex]
   			\hline
   			$\lambda_2=0.5$ & 150  &  $(0.2989,~0.6669)$ & $(0.2358,~0.7309)$  \\[1.2ex]
   			~ & 300 &  $( 0.1891,~0.4607)$ & $(0.1424,~0.5073)$   \\[1.2ex]
   			~& 500  &  $(0.4784,~ 0.6845)$ & $(0.443,~0.7199)$  \\[1.2ex]		  
   			\hline

   			$\lambda_1=0.8$ & 150  &  $( 0.6198,~0.8259)$ & $(0.5844,~0.8613)$  \\[1.2ex]
   			~ & 300  &  $(0.759,~0.9192)$ & $( 0.7315,~0.9467)$  \\[1.2ex]
   			~ & 500  &  $(0.6803,~ 0.8022)$ & $( 0.6593,~0.8231)$  \\[1.2ex]
   			\hline
   			$\lambda_2=0.8$ & 150  &  $(0.7016,~1.0637)$ & $(0.6394,~1.1258)$  \\[1.2ex]
   			~ & 300 &  $( 0.4927,~0.7660)$ & $(0.4457~,0.8129)$   \\[1.2ex]
   			~& 500  &  $(0.7771,~ 0.986)$ & $(0.7412,~1.0219)$  \\[1.2ex]

   			\hline\hline	 		
   	\end{tabular}} 
   
   	\label{tb3} 
   \end{table}

 	        \section{Conclusions}
 	        In this paper, we have studied various informational properties of the family of cubic transmuted distributions. The Shannon entropy and Gini's mean difference of the CT distributions are derived. Further, we have proposed CT Shannon entropy and CT Gini's mean difference.  Connection between  CT Shannon entropy and Kullback-Leibler divergence and that between  CT Gini's mean difference and energy distance have also been established. Chi-square divergence between the CT distribution and the PDFs of its components have been obtained.  Fisher information matrix is obtained for the CT distribution, and we also derive Fisher information for the one-parameter cubic transmuted distribution.  Some simulation results have also been presented to demonstrate the usefulness of some of the informational measures studied here in the context of inference.

\section*{Acknowledgments}
 The authors thank the referees for all their helpful comments and suggestions, which led to the substantial improvements. Shital Saha thanks the University Grants Commission (Award No. $191620139416$), India, for  financial assistantship to carry out this research work. The first two authors thank the research
 	         facilities provided by the Department of Mathematics, National Institute of Technology
 	         Rourkela, India.

\bibliography{refference}

\end{document}